\documentclass[12pt, a4paper, reqno]{amsart}
\usepackage{amssymb}
\usepackage{amsthm}
\usepackage{tikz}
\usetikzlibrary{arrows,positioning} 
\usepackage{mathrsfs}
\usepackage{bbm}
\usepackage{ stmaryrd }
\usepackage[normalem]{ulem}
\usepackage{xfrac}

\usepackage{afterpage}

\usepackage{hyperref}
\hypersetup{
	colorlinks=true,
	linkcolor=blue,
	filecolor=magenta,  
	urlcolor=cyan,
}

\usepackage{hyperref}
\usepackage[capitalise]{cleveref}

\usepackage{amsfonts}
\usepackage{amsthm}
\usepackage{amsmath}
\usepackage{amscd}
\usepackage[latin2]{inputenc}
\usepackage{t1enc}
\usepackage[mathscr]{eucal}
\usepackage{indentfirst}
\usepackage{graphicx}
\usepackage{graphics}
\usepackage{epic}
\numberwithin{equation}{section}
\usepackage[margin=2.9cm]{geometry}
\usepackage{epstopdf} 
\usepackage{subfig}

\usepackage{multirow}
\usepackage{makecell}

\def\End{\operatorname{End}}

\def\ker{\operatorname{ker}}

\def\Rep{\operatorname{Rep}}

\def\Ber{\operatorname{Ber}}

\def\Ad{\operatorname{Ad}}

\def\Spec{\operatorname{Spec}}

\def\C{\mathbb{C}}
\def\Q{\mathbb{Q}}

\def\N{\mathbb{N}}
\def\Z{\mathbb{Z}}
\def\G{\mathbb{G}}

\def\A{\mathbb{A}}

\def\TT{\mathcal{T}}

\def\OO{\mathcal{O}}

\def\FF{\mathcal{F}}

\def\WW{\mathcal{W}}
\def\UU{\mathcal{U}}

\def\NN{\mathcal{N}}

\def\VV{\mathcal{V}}
\def\II{\mathcal{I}}

\def\a{\mathfrak{a}}
\def\b{\mathfrak{b}}
\def\c{\mathfrak{c}}
\def\m{\mathfrak{m}}

\def\p{\mathfrak{p}}
\def\q{\mathfrak{q}}
\def\g{\mathfrak{g}}

\def\t{\mathfrak{t}}
\def\h{\mathfrak{h}}

\def\k{\mathfrak{k}}
\def\l{\mathfrak{l}}
\def\s{\mathfrak{s}}
\def\o{\mathfrak{o}}

\def\d{\partial}

\def\ol{\overline}

\def\sub{\subseteq}

\newtheorem{thm}{Theorem}[section]
\newtheorem{cor}[thm]{Corollary}
\newtheorem{lemma}[thm]{Lemma}
\newtheorem{prop}[thm]{Proposition}

\theoremstyle{definition}
\newtheorem{definition}[thm]{Definition}
\theoremstyle{remark}
\newtheorem{remark}[thm]{Remark}
\newtheorem{example}[thm]{Example}

\numberwithin{equation}{section}

\usepackage{relsize}
\usepackage{fancyhdr}
\usepackage[all,cmtip]{xy}

\begin{document}
	
	\title{Localization theorem for homological vector fields}
	
	\author{Vera Serganova, Alexander Sherman}

	\begin{abstract} We present a general theorem which computes the cohomology of a homological vector field on global sections of vector bundles over smooth affine supervarieties.  The hypotheses and results have the clear flavor of a localization theorem. 
	\end{abstract}
	
	\maketitle
	\pagestyle{plain}

	\section{Introduction}
	
	Let $X$ be an affine supervariety, and let $Q$ be an odd vector field on $X$; that is, an odd derivation of its algebra of functions, $\Bbbk[X]$.  We call $Q$ \emph{homological} if $Q^2$ acts semisimply on $\Bbbk[X]$.  As the name suggests, taking the cohomology of a homological vector field, i.e.~$\ker(Q:\Bbbk[X]^{Q^2})/\operatorname{Im}(Q:\Bbbk[X]^{Q^2})$, is a `nice' operation.  Here $\Bbbk[X]^{Q^2}$ denotes the $Q^2$-invariants of $\Bbbk[X]$.  We write this cohomology, which is a super vector space, as $DS_Q\Bbbk[X]$, because in the case when $Q$ comes from the action of a Lie superalgebra $\g$, it is exactly the value of the Duflo-Serganova functor, from the theory of Lie superalgebras.  The goal of the present paper is to understand conditions under which we may localize the computation of $DS_Q$ to a smooth, closed subvariety $Y$ of $X$.
	
	The Duflo-Serganova functor has played a prominent role in the representation theory of Lie superalgebras since being introduced in \cite{duflo2005associated} (see, for instance, \cite{heidersdorf2014cohomological}, \cite{serganova2011superdimension}, \cite{gorelik2020semisimplicity}, \cite{entovaaizenbud2019dufloserganova}, \cite{entova2018deligne}).  For an up-to-date survey dealing with the Duflo-Serganova functor in the representation theory of Lie superalgebras, we refer to \cite{GHSS}.
	
	On the other hand, localization theorems have proven to be powerful tools in computations of integrals and cohomology, providing vast generalizations of previously understood results.  As an important example for us, in \cite{schwarzzaborsupersymlocal} a localization theorem was proven for computing Berezin integrals on a supermanifold when the volume form admits an odd symmetry satisfying sufficiently nice properties.  This was recently used in \cite{SS22} and \cite{SV} to prove `splitting 'properties of distinguished subgroups of certain (almost) simple supergroups, which analogizes the definition of Sylow subgroups in the finite characteristic setting. In this paper, our hypotheses are not too far from the ones of \cite{schwarzzaborsupersymlocal}, except that we work in the algebraic category.

	More precisely, suppose that $X$ is a smooth affine supervariety with odd vector field $Q$, such that $Q^2$ acts semisimply on $\Bbbk[X]$.  Then if $Q$ vanishes on a smooth subvariety $Y$, such that it acts in a certain nondegenerate way on the normal bundle of $Y$ in $X$, we show that $DS_Q\Bbbk[X]=\Bbbk[Y]$.  One can view this theorem as a generalization of a result which is already known: that if $Q$ is non-vanishing on all of $X$, then $DS_Q\Bbbk[X]=0$.
	
	The precise theorem (\cref{main_thm}) is slightly more general, both in that it begins with a vector bundle on $X$, and it weakens the assumption that $Q$ needs to vanish on all of $Y$.  In any case, the result provides a powerful tool for computations, many of which are of great interest.  We attempt to survey some of the applications in Sections 5 and 6.
	
	As will be seen in the proof, the main technical input is from the theory of commutative algebra; namely, the vector fields we consider have the property that locally they define a Koszul complex for a sequence of regular elements in a regular commutative ring.  Thus by using what is already known about the Koszul complex in such a setting, we are able to prove the vanishing results for the cohomology of our operator.
	
	\begin{remark}
		It is clear that one can take the cohomology 
		\[
		\ker(Q:\Bbbk[X]^{Q^2})/\operatorname{Im}(Q:\Bbbk[X]^{Q^2}),
		\]
		without the assumption that $Q^2$ acts semisimply; however this assumption is crucial for our results, as it implies that we are working with the action of a Lie superalgebra which has a `nice' (Frobenius) category of representations, in which case $DS_Q$ has the interpretation of a semisimplification functor.  Thus we have enough projective-injective modules in our category, which is used throughout. See Section 3 for more details.
	\end{remark}
	
	\subsection{Relation to physics} The Duflo-Serganova functor has appeared under different guises in theoretical physics.  Odd operators $Q$ in a supersymmetric field theory that satisfy $Q^2=0$ are examples of BRST operators, and allow one to employ the so-called BRST formalism, which includes taking cohomology in $Q$.  Such a situation arises in several places in the literature.  In \cite{W}, it was used as a key part in topological and holomorphic twists of  supersymmetric field theories.  Twisting a supersymmetric field theory gives rise to simpler field theories that can be topological (giving a TQFT), holomorphic, or something in between, depending on properties of the chosen $Q$.  A mathematically rigorous approach to the twisting of supersymmetric field theories is explained in \cite{Cos}. 
	
	Note that our setting is more flexible in that we do not impose the requirement that $Q^2=0$; however we show in \cref{cor_assume_sq_zero} that one can always reduce to the case when $Q^2=0$.
	
	In \cite{CCMV}, the authors also consider the cohomology of homological vector fields (referred to as cohomological reductions in their paper) on the smooth (or $L^2$) sections of certain vector bundles on the target spaces of sigma models, which are homogeneous superspaces. In contrast, in the algebraic setting we find unexpected extra components of these cohomological reductions; see for instance the table in \cref{thm table} for the case of symmetric spaces.  In fact, when our space is a group $G$ and we consider the action of an adjoint vector field, we find extra odd infinitesimal symmetries in the cohomological reduction, which will themselves act on the Duflo-Serganova functor (see \cref{thm G_u comp}).
	
	We finally highlight the connection of our work to \cite{GK}, where the cohomology of a vector field on $Q$-bundles is applied to the study of gauge PDEs.  There, they define a notion of equivalent reduction of supermanifolds with a BRST operator $Q$ as ones which differ from one another by a `$Q$-contractible' fiber bundle, i.e. a $Q$-bundle with trivial $Q$-cohomology along the fibers.  In this case they obtain that the $Q$-cohomology of supermanifolds with equivalent reductions is the same.
	
	\subsection{Outlook for non-affine varieties}  Our localization theorem can be extended to the non-affine case, although one must work in the derived category to obtain the correct framework.  This will be left to a future work, where we will in particular study applications to vector bundles on flag varieties.
	
	\subsection{Summary of sections} In Section 2 we recall the basic supergeometric language we will use, and some preliminary results.  In Section 3 we discuss $\q$-supervarieties and introduce the fundamental tools we need for our main theorem.  Section 4 contains the main theorem and its proof.  Section 5 discusses the first applications, to the computation of the supergroup $\widetilde{G_u}$.  Here we also prove an important criterion for identifying a supervariety as the odd tangent bundle of an even variety.  In Section 6 we discuss computations on certain homogeneous superspaces, in particular supersymmetric spaces.  Section 7, an appendix, computes $\widetilde{G_u}$ for a certain root subgroup of $D(2,1;\alpha)$.
	
	\subsection{Acknowledgments}  The authors are grateful to I.~Entova-Aizenbud, V.~Hinich, and D. Vaintrob for many helpful discussions.   The first author was supported in part by NSF grant 2001191 and by Tromso research foundation. The second author was partially supported by ISF grant 711/18, NSF-BSF grant 2019694, and by ARC grant DP210100251.
	
	\section{Supergeometric preliminaries}\label{supergeo_sec}
	
	We work throughout over an algebraically closed field $\Bbbk$ of characteristic zero.  For a super vector space $V$ we write $V=V_{\ol{0}}\oplus V_{\ol{1}}$ for its parity decomposition. If $v\in V$ is homogeneous, we write $\ol{v}\in\{\ol{0},\ol{1}\}=\Z/2\Z$ for its parity.  If $R$ is a supercommutative ring, then we denote by $R^{m|n}$ a free $R$-module of rank $m|n$.  If $M$ is a free $R$-module of finite rank $(m|n)$, then we may define $\Ber(M)$, the Berezinian of $M$, as in Sec.~3.3.7 of \cite{M}.  It is again a free $R$-module of rank $(1|0)$ if $n$ is even and rank $(0|1)$ if $n$ is odd.
	
	\subsection{Notation for supervarieties}\label{section_notation} We will use the symbols $X,Y,\dots$ for supervarieties with even subschemes $X_0,Y_0,\dots.$  For a supervariety $X$ we write $\OO_X$ for the structure sheaf and $\Bbbk[X]:=\Gamma(X,\OO_X)$ for the superalgebra of global functions on $X$.  We will be considering supervarieties in the sense of Def.~2.1 of \cite{sherman2021spherical}, however all spaces of interest will be smooth and affine.  A smooth affine supervariety is always given by the exterior algebra of a finite rank vector bundle on a smooth affine variety $X_0$ (see \cite{voronov1990elements}).  For several equivalent conditions to being smooth, which we will use at various points without elaboration, see the appendix of \cite{sherman2021spherical}.
	
	Note that affine supervarieties and morphisms between them are entirely determined by their algebras of global functions and maps between them, just as with affine varieties.  Further, we note that the space of global functions of an affine supervariety is finitely generated as a $\Bbbk$-superalgebra.  See \cite{carmeli2011mathematical} for more on the basics of algebraic supergeometry.
	
	We will employ the convention that for functions on a supervariety the symbols $t,s$ will denote even functions, $\xi,\eta$ odd functions, and $f,g$ functions of ambiguous parity.
	
	If $X$ is a supervariety, there is a canonical closed embedding $i_X:X_0\to X$ which is a homeomorphism of underlying topological spaces.  The closed points of $X$ are the $\Bbbk$-points, which we write as $X(\Bbbk)$, and they are canonically identified with the closed points of $X_0$ via $i_X$.  If $x$ is a closed point of $X$ and $\FF$ is coherent a sheaf on $X$, we write $\FF_x$ for the stalk of $\FF$ at $x$, and $\FF|_x$ for the fiber of $\FF$ at $x$.    We will write $\m_x$ for the maximal ideal of $\Bbbk[X]$ corresponding to $x$.  Then the cotangent space at $x$ is given by $T_x^*X:=\m_x/\m_x^2$, and the tangent space by $T_xX=\left(T_x^*X\right)^*$.
	
	A vector bundle on a supervariety $X$ is by definition a coherent sheaf $\VV$ which is locally isomorphic to $\OO_{X}^{m|n}$ for some $m,n\in\N$.  We define $\Ber(\VV)$, the Berezinian of $\VV$, to be the rank one vector bundle obtained by glueing together $\Ber(\VV|_{U})$ on open sets $U$ for which $\VV|_U$ is trivialized.
	
	Given a function $f\in\Bbbk[X]$, we will use the notation $D(f)$ of Chapt.~2.2 of \cite{hartshorne2013algebraic} for the open subvariety on which $f$ is non-vanishing.
	
	Finally, if $Y$ is a closed subvariety of $X$, we write $\II_Y$ for the ideal sheaf of $Y$ in $X$, and $\NN_{Y}^\vee=\II_Y/\II_Y^2$ for the conormal bundle of $Y$ in $X$.
	
	\subsection{Coordinates} Let $X$ be a smooth supervariety.  We say functions $t_1,\dots,t_m,\xi_1,\dots,\xi_n$ define a coordinate system on an open set $U$ in $X$ if they are regular on $U$ and $dt_1,\dots,dt_m,d\xi_1,\dots,d\xi_n$ trivialize the cotangent bundle on $U$.  In this case, the tangent bundle will be trivialized on $U$ by the sections $\d_{t_1},\dots\d_{t_m},\d_{\xi_1},\dots,\d_{\xi_n}$ which are derivations acting in the obvious way.  
	
	By general facts about sections of vector bundles, if we have functions $t_1,\dots,t_m,\xi_1,\dots,\xi_n$ which define linearly independent differentials at a point, then they must define linearly independent differentials in a neighborhood of that point; in particular they define a coordinate system on a neighborhood of said point.
	
	We will use the following lemma later on:
	\begin{lemma}\label{kers define splitting}
		Suppose that $X=\Spec A$ is a smooth affine supervariety and $\xi_{1},\dots,\xi_{n}\in A_{\ol{1}}$ are global odd coordinates on $X$.  Then the derivations $\d_{\xi_{1}},\dots,\d_{\xi_n}$ are well-defined, and if we set $A_0=\ker\d_{\xi_1}\cap\cdots\cap\d_{\xi_{n}}$, the natural map
		\[
		A_0[\xi_1,\dots,\xi_n]\to A
		\]
		is an isomorphism.
	\end{lemma}
	
	\begin{proof}
		Because $X$ is smooth, the cotangent bundle $\Omega_{X}$ is locally free.  The global odd coordinates $\xi_1,\dots,\xi_n$ define a map $\OO_{X}^{0|n}\to\Omega_{X}$ which is injective on fibers, and thus splits because $X$ is affine.  This splitting provides for us our derivations $\d_{\xi_i}$ for all $i$.  Now let $A_0=\ker\d_{\xi_1}\cap\cdots\cap\d_{\xi_{n}}$, and consider the natural map
		\[
		A_0[\xi_1,\dots,\xi_n]\to A.
		\]
		To see this map is injective, suppose that $\sum\limits_{I}f_I\xi_I=0$ where $f_I\in A_0$.  Then if we apply $\d_I:=\prod_{i\in I}\d_{\xi_{i}}$ to $f$ we learn that $f_I=0$ for all $I$.  
		
		To show surjectivity, first order the subsets $I\sub\{1,\dots,n\}$ using a lexicographical ordering, meaning that $I=\{i_1,\dots,i_r\}<J=\{j_1,\dots,j_s\}$ whenever $|I|<|J|$, and if $|I|=|J|$ then  $I<J$ if for some $k$ we have $i_k<j_k$ while $i_\ell=j_\ell$ for $\ell<k$.  
		
		Now let $f\in A$ and let $I$ be the largest subset under this ordering such that $\d_I(f)\neq0$.  Then we must have that $\d_{\xi_i}(\d_I(f))=0$ for all $i$, i.e. $\d_I(f)=g_I\in A_0$.  Consider 
		\[
		f':=f-\d_I(\xi_I)g_I\xi_I
		\]
		Then $\d_I(f')=0$ by construction, and thus we may conclude by induction that $f'\in A_0[\xi_1,\dots,\xi_n]$, completing the proof.
	\end{proof}
	
 	\subsection{Subvarieties}  By abuse of language, we will refer to subsupervarieties as subvarieties.  If $Y\sub X$ is a closed subvariety of $X$, then we write $\II_Y$ for the ideal sheaf of $Y$, and $\NN_Y^\vee:=\II_Y/\II_Y^2$ for the conormal sheaf of $Y$.  Then  $\NN_Y^\vee$ is a coherent sheaf on $Y$, and if $Y$ is smooth then $\NN_Y^\vee$ is a vector bundle of rank given by the codimension of $Y$.  In this case, every point on $Y$ has an open neighborhood in $X$ where we may choose local functions $t_1,\dots,t_m,\xi_1,\dots,\xi_n\in\II_Y$ such that they will trivialize  $\NN_Y^\vee$, and they define linearly independent differentials.
	
	\section{\texorpdfstring{$\q$}{q}-supervarieties}
	
	We explain several lemmas and definitions we will use repeatedly for the proof of the general theorem.   
	\subsection{\texorpdfstring{$\q$}{q}-modules} Let $\q$ be the $(1|1)$-dimensional Lie superalgebra with basis $Q,Q^2:=\frac{1}{2}[Q,Q]$, where $Q$ is odd.  We will work in $\Rep_{Q^2}\q(1)$, i.e. the category of $\q$-modules (not necessarily finite-dimensional) on which $Q^2$ acts semisimply.  Thus when we refer to a $\q$-module we will always assume that $Q^2$ acts semisimply.  
	
	\begin{definition}[Duflo-Serganova functor]
		For a $\q$-module $V$ (recall that as explained above, this mean $Q^2$ acts semisimply), define $DS_QV$ to be the super vector space given by the cohomology of $Q$ when restricted to $V^{Q^2}$, the invariants (equivalently kernel) of $Q^2$ on $V$.   We will also write this occasionally using the shorthand $V_Q:=DS_{Q}V$.
	\end{definition}
	The Duflo-Serganova functor was originally introduced in \cite{duflo2005associated}, and defines a tensor functor from $\Rep_{Q^2}\q$ to the category of super vector spaces.  It is in fact a semisimplification functor of $\Rep_{Q^2}\q(1)$, although we will not use this.  For a more recent and thorough survey on this functor, we refer to \cite{GHSS}.  
	
	We recall the following results which will be used throughout.
	\begin{lemma}\label{q mod proj criteria} Let $V$ be a $\q$-module.  
		\begin{enumerate}
			\item 	$V$ is projective if and only if $DS_QV=0$;
			\item 	if $V$ admits a finite-length projective resolution, then $V$ is projective;
			\item 	if the restriction $Q:V_{\ol{1}}\to V_{\ol{0}}$ is an isomorphism, then $V$ is projective.
		\end{enumerate}
	\end{lemma}
	
	\subsection{\texorpdfstring{$\q$}{q}-supervarieties}\label{section_q_supervars}  
		We recall that on an affine supervariety $X=\Spec\Bbbk[X]$, a vector field is nothing but a derivation (in the super sense) of $\Bbbk[X]$.
	\begin{definition}		
	An affine $\q$-supervariety is a pair $(X,Q)$, where $X$ is an affine supervariety and $Q$ is an odd vector field on $X$ such that $Q^2$ acts semisimply on $\Bbbk[X]:=\Gamma(X,\OO_X)$.  In other words, $\Bbbk[X]$ is a $\q$-module where $Q$ acts by derivations. In the future we will simply write that $X$ is a $\q$-supervariety, supressing the $Q$.
	\end{definition}
	
    Let $X$ be a $\q$-supervariety.  Then since $\Bbbk[X]$ is finitely generated, the $\q$-action action necessarily integrates to the action of an algebraic supergroup $G_Q$ on $X$, where $(\operatorname{Lie}G_Q)_{\ol{1}}$ is spanned by $Q$, and $(G_Q)_0$ is a central torus (of some finite dimension).  Because of the ambiguity in this choice of torus depending on the representation, it is better to avoid fixing a lift to $G_Q$ and instead only doing so when it is needed in an argument. 
    
    \begin{definition}
    Let $X$ be a $\q$-supervariety.  We say that an open subvariety $U\sub X$ is $\q$-stable if $Q^2$ acts semisimply on $\Bbbk[U]$.  In this case $U$ is itself naturally a $\q$-supervariety.
    \end{definition}

	\begin{definition}
		Given a vector field $Q$ on $X$ and a point $x\in X(\Bbbk)$, we may evaluate $Q$ at $x$ to obtain a tangent vector $Q_x\in T_{x}X$.  We say $Q$ is non-vanishing at $x$ if $Q_x\neq0$, and otherwise say it is vanishing at $x$.  We say $Q$ is non-vanishing on $X$ if it is non-vanishing at every point $x\in X(\Bbbk)$. 	We write $Z:=Z(Q)$ for the closed subvariety of $X_0$ given by the vanishing set of $Q$.  
	\end{definition}
	
	\begin{remark}
		Observe that if $\II_Z$ denotes the ideal sheaf defined by $Z:=Z(Q)\sub X$, we have $Q(\OO_X)\sub\II_Z$.
	\end{remark}
	
	\subsubsection{Equivariant sheaves} Let $\FF$ be a quasi-coherent sheaf on a $\q$-supervariety $X$.  We say that $\FF$ is $\q$-equivariant if it admits an action of $Q$ such that $Q^2$ acts semisimply on global sections, and 
	\[
	Q(fv)=Q(f)v+(-1)^{\ol{f}}fQ(v),
	\]
	where $v$ is a section of $\FF$ and $f$ is a function.  
	
	\begin{lemma}\label{non_vanishing_lemma}
          Let $X$ be an affine $\q$-supervariety; then the following are equivalent:
          \begin{enumerate}
          	\item $Z(Q)=\emptyset$;
          	\item there exists $\xi\in\Bbbk[X]$ with $Q(\xi)=1$;
          	\item $DS_Q\Bbbk[X]=0$;
          	\item if $\FF$ is any quasi-coherent $\q$-equivariant sheaf on $X$, then $DS_Q\Gamma(X,\FF)=0$.
          \end{enumerate}  
	\end{lemma}

	\begin{proof}
		Clearly $(4)\Rightarrow(3)$, $(3)\Rightarrow (2)$, and $(2)\Rightarrow (1)$.  Further $(2)\Rightarrow (4)$ because if $s\in\Gamma(X,\FF)$ with $Q(s)=0$, then $Q(\xi s)=s$.
		
		Thus it suffices to show $(1)\Rightarrow(2)$. Since $Q$ is everywhere non-vanishing and $X$ is affine, for each closed point $x\in X(\Bbbk)$ there exists an odd function $\xi_x\in \Bbbk[X]$ such that $(Q(\xi_x))(x)\neq0$.  Thus the ideal generated by $\{Q(\xi_x)\}_{x\in X(\Bbbk)}$ is equal to $\Bbbk[X]$, so we may choose finitely many functions in this set, say $Q(\xi_1),\dots,Q(\xi_k)$, and even functions $g_1,\dots,g_k$ such that
		\[
		g_1Q(\xi_1)+\dots+g_kQ(\xi_k)=1.
		\] 
		Set $\eta=g_1\xi_1+\dots+g_k\xi_k$.  We have
		\[
		Q\eta=[g_1Q(\xi_1)+\dots+g_kQ(\xi_k)]+[Q(g_1)\xi_1+\dots+Q(g_k)\xi_k]=1+r
		\]
		where $r:=Q(g_1)\xi_1+\dots+Q(g_k)\xi_k$.  Since $Q(g_i)$ is odd, $r\in(\Bbbk[X]_{\ol{1}})^2$ and thus is nilpotent.  We may then write $\eta=\sum\limits_{\lambda}\eta_{\lambda}$ and $r=\sum\limits_{\lambda}r_{\lambda}$ according to the decomposition of $\Bbbk[X]$ into eigenspaces of $Q^2$.  Since $Q^2$ preserves $(\Bbbk[X]_{\ol{1}})^2$, we have $r_{\lambda}\in(\Bbbk[X]_{\ol{1}})^2$ for all $\lambda$.  Now we see that
		\[
		Q(\eta)=\sum\limits_{\lambda}Q(\eta_{\lambda})=1+\sum\limits_{\lambda}r_{\lambda}.
		\]
		Since $Q^2(1)=0$ and $[Q^2,Q]=0$, we find that $a:=Q(\eta_{0})=1+r_0$ is a unit in $\Bbbk[X]$.  Further, $Q(a)=Q^2(\eta_0)=0$, so if we set $\xi:=\eta_0/a$, we find that
		\[
		Q(\xi)=Q(\eta_0/a)=a/a=1.
		\]	
	\end{proof}

	The following lemma is one of the main technical inputs in the proof of the main theorem in the next section.
	\begin{lemma}\label{fixed_point_neighborhood}
		Let $\WW$ be a $\q$-equivariant vector bundle on a smooth affine $\q$-supervariety $X$, and suppose that $Q$ vanishes at $x\in X(\Bbbk)$.   Then there exists an affine open $\q$-stable neighborhood $U$ of $x$ such that we have a $Q^2$-equivariant isomorphism of sheaves 
		\[
		\WW|_{U}\cong\OO_U\otimes_{\Bbbk} \WW|_x.
		\]
		If we suppose further that $\WW\to\WW|_x$ admits a $\q$-splitting, then the isomorphism can be made to be $\q$-equivariant.  In particular in this case, $\Gamma(U,\WW)$ is a projective $\q$-module if $\WW|_x$ is.
	\end{lemma}
	
	\begin{proof}
		Since $Q^2$ acts semisimply we may split $\Gamma(X,\WW)\to\WW|_x$ as a $Q^2$-module in any case; choose such a splitting, and if a $\q$-equivariant splitting exists then choose the splitting to be $\q$-equivariant.  Then we have a homogeneous basis $v_1,\dots,v_m,w_1,\dots,w_n$ of our assumed splitting, i.e. coming from the image of $\WW|_x$ under the splitting map, such that each $v_i,w_j$ is a global section and a $Q^2$-eigenvector.
		
		Let $U$ be the maximal open subset on which $v_1,\dots,v_m,w_1,\dots,w_n$ trivialize $\WW$; then $U$ is affine, as it is the complement of the divisor determined by a rational, $Q^2$-eigensection of $\Ber(\WW)$ (see Section \ref{section_notation} for notation).  Since the rational section is a $Q^2$-eigenvector, $\Bbbk[U]$ will admit a semisimple action of $Q^2$.  Finally, since $\WW$ is trivialized on $U$ by these sections, the natural map $\OO_U\otimes_{\Bbbk}\WW|_x\to\WW|_U$ is an isomorphism.
	\end{proof}
	
	\begin{lemma}\label{point_neighborhood_redn}
		Suppose that $\FF$ is a $\q$-equivariant quasicoherent sheaf such that every point $x\in X(\Bbbk)$ has a $\q$-stable, affine open  neighborhood $U$ such that for every $\q$-stable affine open $V\sub U$, we have that $\Gamma(V,\FF)$ is projective over $\q$.  Then $\Gamma(X,\FF)$ is a projective $\q$-module.
	\end{lemma}
	
	\begin{proof}
		Using our assumptions we may find a $\q$-stable affine Cech cover $\{U_i\}$ of $X$ such that $\Gamma(U_{i_1}\cap\cdots\cap U_{i_k},\FF)$ is projective for all $i$.  Since $X$ is affine, the Cech complex 
		\[
		0\to \Gamma(X,\FF)\to \prod\limits_{i}\Gamma(U_i,\FF)\to\prod\limits_{i,j}\Gamma(U_i\cap U_j,\FF)\to\cdots
		\]
		will be exact and finite, so we can conclude by (2) of Lemma \ref{q mod proj criteria}.
	\end{proof}

	\begin{cor}\label{cor_vec_bndl}
	Suppose that $\WW$ is a $\q$-equivariant vector bundle on a smooth affine $\q$-supervariety $X$ such that at every point $x\in Z(Q)$, $\WW|_x$ is a projective $\q$-module.  Then $\Gamma(X,\WW)$ is a projective $\q$ module.
	\end{cor}
	\begin{proof}
		We apply \cref{point_neighborhood_redn}.  If $x\in X(\Bbbk)\setminus Z(Q)$, then we may find a $\q$-stable, affine open neighborhood $U$ of $x$ such that $Q$ is nonvanishing on $U$.  Thus by \cref{non_vanishing_lemma}, $DS_Q\Gamma(V,\WW)=0$ for any $\q$-stable affine open $V\sub U$. 
		
		If on the other hand $x\in Z(Q)$, then by \cref{fixed_point_neighborhood} there exists a $\q$-stable, affine open neighborhood $U$ of $x$ such that $\WW|_U\cong \OO_U\otimes\WW|_x$.  Therefore, for any $\q$-stable affine open $V\sub U$ we will have $\WW|_V\cong\OO_V\otimes\WW|_x$, so since $\WW|_x$ is projective, $\Gamma(V,\WW)$ is also projective, and we are done.
	\end{proof}

	In the following example we show that the converse of \cref{cor_vec_bndl} fails.
	\begin{example}
		Let $X=\A^{1|1}$ and consider it with a trivial action of $\q$; in particular $Z(Q)=X(\Bbbk)$.  Present $\Bbbk[X]=\Bbbk[t,\xi]$, and give $\OO_X$ the structure of a $\q$-equivariant sheaf via $Q(f)=\xi f$.  Then we have $DS_Q\Gamma(X,\OO_X)=0$, even though the
		fibers of this vector bundle are not projective as $\q$-modules.
	\end{example}

	\section{Main theorem}
	
	\begin{thm}\label{main_thm}
		Let $\VV$ be a $\q$-equivariant vector bundle on a smooth affine $\q$-supervariety $X$.  Write $Z=Z(Q)\sub X_0$ for the vanishing subvariety of $Q$, and suppose that $Y$ is a smooth, closed subvariety of $X$ such that:
		\begin{enumerate}
			\item $Z(\Bbbk)\sub Y(\Bbbk)$;
			\item $Q(\II_Y)\sub\II_Y$;
			\item $Q:(\NN_Y^\vee|_z)_{\ol{1}}\to(\NN_Y^\vee|_{z})_{\ol{0}}$ is an isomorphism for all $z\in Z(\Bbbk)$.
			\item For all points $z\in Z(\Bbbk)$, there exists a $\q$-stable affine open $U_z$ of $z$ such that the map $\Gamma(U_z,\VV)\to\VV|_{z}$ splits over $\q$.
		\end{enumerate}
		Then $Y$ is naturally a $\q$-supervariety, $\VV|_Y$ is a $\q$-equivariant vector bundle on $Y$, and we have: 
		\[
		DS_Q\Gamma(X,\VV)=DS_Q\Gamma(Y,\VV|_{Y}).
		\]
		Further this identification is induced by the natural map $\Gamma(X,\VV)\to \Gamma(Y,\VV|_Y)$.
	\end{thm}

	\subsection{Corollaries of Theorem \ref{main_thm}} Note that if $\VV=\OO_X$ is the trivial vector bundle with the canonical $\q$-equivariant structure, then condition (4) always holds.  Thus we obtain:
	\begin{cor}
		Under hypotheses (1)-(3) of Theorem \ref{main_thm}, we have that the restriction map $\Bbbk[X]\to \Bbbk[Y]$ induces an identification:
		\[
		DS_Q \Bbbk[X]=DS_Q \Bbbk[Y].
		\]
	\end{cor}

	\begin{cor}\label{cor_z=y}
		Suppose that in the hypotheses of Theorem \ref{main_thm} we have  $Q(\OO_X)\sub\II_Y$.  Then the restriction map $\Bbbk[X]\to \Bbbk[Y]$ induces an identification
		\[
		DS_Q\Bbbk[X]=\Bbbk[Y].
		\]
		Further, suppose that $\VV$ is a $\q$-equivariant vector bundle on $X$ satisfying (4) of \cref{main_thm}.  Then $DS_Q\Gamma(X,\VV)$ will be the sections of a vector bundle on $Y$ with fibers isomorphic to $DS_Q(\VV|_y)$.
	\end{cor}

	\begin{proof}
		The first statement follows immediately from \cref{main_thm}.  For the second statement, we may apply \cref{fixed_point_neighborhood} to show that $Q$ defines a constant rank vector bundle endomorphism of $\VV|_Y$, whose cohomology will have fibers as claimed.
	\end{proof}
	
	By the following corollary, \cref{main_thm} allows us to always reduce to the case when $Q^2=0$.
	\begin{cor}\label{cor_assume_sq_zero}
		If $Y=\Spec\Bbbk[X]/(\operatorname{Im}Q^2)$ where $(\operatorname{Im}Q^2)$ is the ideal generated by $\operatorname{Im}Q^2$, then $Y$ is a smooth subvariety of $X$ and we have
		\[
		DS_Q\Bbbk[X]=DS_Q\Bbbk[Y].
		\]
	\end{cor}
	\begin{proof}
	We need to check that $Y$ is smooth and that $Q:(\NN_Y^\vee|_z)_{\ol{1}}\to(\NN_Y^\vee|_{z})_{\ol{0}}$ is an isomorphism for all $z\in Z(\Bbbk)$.  Let $I=(\operatorname{Im}Q^2)$, and let $y\in Y(\Bbbk)$.  Choose a coordinate system $t_1,\dots,t_m,\xi_1,\dots,\xi_n$ around $y$ such that these coordinates are eigenvectors. Then we claim that $I_y$, the localization of $I$ at $y$, is generated by those coordinates with nonzero eigenvalues.  The proof of this statement is identical to the one given in Lem. 2.1 of \cite{I72}, via Prop. 3.2 of \cite{Sch89}.  Thus we obtain that $Y$ is smooth, and further we have that $Q^2$ is an automorphism of $\NN_Y^\vee|_{y}$ for all $y\in Y(\Bbbk)$, which is sufficient for condition (3) of \cref{main_thm}.
\end{proof}
	
	\subsection{Proof of Theorem \ref{main_thm}}  The proof of Theorem \ref{main_thm} will occupy the remainder of this section.  In fact, Theorem \ref{main_thm} follows as corollary of the following result, which is what we will actually show.
	\begin{prop}\label{main_prop}
		Retain the hypotheses of Theorem \ref{main_thm}.  Let $\VV_Y$ denote the sheaf of sections of $\VV$ which vanish upon restriction to $Y$.  Then $\Gamma(X,\VV_Y)$ is a projective $\q$-module; in particular $DS_Q\Gamma(X,\VV_Y)=0$. 
	\end{prop}
	
	Indeed, Proposition \ref{main_prop} is enough because then if we apply $DS_Q$ to
	\[
	0\to \Gamma(X,\VV_Y)\to\Gamma(X,\VV)\to\Gamma(Y,\VV|_Y)\to 0,
	\]
	we obtain that $DS_Q\Gamma(X,\VV)\to DS_Q\Gamma(Y,\VV|_Y)$ is an isomorphism, see Lemma 2.7 of \cite{GHSS}.
	
	\begin{proof}[Proof of Proposition \ref{main_prop}] We first deal with the case of $\VV=\OO_X$, and use Lemma \ref{point_neighborhood_redn}:
		
	\textbf{Points where $Q$ doesn't vanish:} If $x\in X(\Bbbk)\setminus Z(\Bbbk)$, then by affinity and the semisimplicity of the action of $Q^2$, there exists a $Q^2$-eigenfunction $f\in\Gamma(X,\II_Z)$ such that $f(x)\neq0$.  Thus $D(f)$ will be a $\q$-stable affine open subvariety containing $x$ on which $Q$ is nonvanishing.  Thus, by Lemma \ref{non_vanishing_lemma}, $\Gamma(V,\VV_Y)$ is projective for all $\q$-stable affine open subsets $V\sub D(f)$.

	\textbf{Finding convenient local coordinates where $Q$ does vanish:} It remains to deal with the points in $Z(\Bbbk)$.  For this step we begin by assuming that $\VV=\OO_X$.  Choose $z\in Z(\Bbbk)$. We want to find a `nice' splitting of $\m_z\to\m_z/\m_z^2$; to expect a full $\q$-module splitting would be too much, but we can obtain enough as follows.  Recall that $\NN_Y^\vee|_z$ is a subspace of $\m_z/\m_z^2$.
	
	We have a natural, $\q$-equivariant surjective map $\Gamma(X,\II_Y)\to\NN^\vee_Y|_z$; by our assumptions on $\NN^\vee_Y|_z$ it is a projective $\q$-module, and so we may split this map $\q$-equivariantly; in particular from this splitting we may choose sections $t_1,\dots,t_n,\xi_1,\dots,\xi_n\in\Gamma(X,\II_Y)$ satisfying:
	\[
	Q(\xi_i)=t_i, \ \ \ \ Q(t_i)=c_i\xi_i,
	\]
	where $c_i\in \Bbbk$.  Further, by the discussion at the end of \cref{supergeo_sec}, these functions define linearly independent differentials in a neighborhood of $z$.  Write $V$ for the vector space spanned by $t_1,\dots,t_n,\xi_1,\dots,\xi_n$.  Then the natural map $V\to T_z^*X$ is injective and $\q$-equivariant.  Choose a $Q^2$-equivariant splitting $T^*_zX=V\oplus W$, and a $Q^2$-equivariant splitting of $W$ off $\m_z\to T_z^*X$, thus giving us functions 
	\[
	t_1,\dots,t_n,s_1,\dots,s_r,\xi_1,\dots,\xi_n,\eta_1,\dots,\eta_\ell\in\m_z
	\] 
	each of which is a $Q^2$-eigenvector (here the $s_i$'s are even and the $\eta_i$'s are odd).  These functions define linearly independent differentials in a neighborhood of $z$; we may now apply the proof of \cref{fixed_point_neighborhood} using this splitting, to the case where $\WW$ is the cotangent bundle, to find a $\q$-stable, affine open neighborhood $U'$ of $z$ on which these functions continue to define linearly independent differentials.
	
	\textbf{Cutting down $U'$ to $U$ so that $Y\cap U=(t_1=\cdots=t_n=\xi_1=\cdots\xi_n=0)$:} Consider the ideal $J:=(t_1,\dots,t_n,\xi_1,\dots,\xi_n)\sub\Bbbk[U']$; then we need further that $J=I_{Y\cap U'}$, i.e. that the ideal $J$ exactly cuts out $Y\cap U'$.  A priori we only have $J\sub I_{Y\cap U'}$.  However since $t_1,\dots,t_n,\xi_1,\dots,\xi_n$ are part of a coordinate system, $J$ will cut out a smooth subvariety of the same dimension as $Y$, and thus the zero set of $J$ can be written as $Y\sqcup Y'$, i.e. $Y$ will be a connected component of its zero set.
	
	Consider the ideal $I_{Y'}$ of $Y'$ in $\Bbbk[U']$.  Because $J$ is $Q^2$-stable, $I_{Y'}$ is also $Q^2$-stable; by affinity there must exist a $Q^2$-eigenfunction $f\in I_{Y'}$ such that $f(z)\neq0$.  Now we set $U:=D(f)\sub U'$; then $U$ is a $Q^2$-stable affine open neighborhood of $z$, and $I_{Y\cap U}=(t_1,\dots,t_n,\xi_1,\dots,\xi_n)$.
	
	\textbf{Form of $Q$ on $U$:} Now on $U$ we have a trivialization of the tangent bundle defined by the sections $\d_{t_1},\dots,\d_{t_n},\d_{s_1},\dots,\d_{s_r},\d_{\xi_1},\dots,\d_{\xi_n},\d_{\eta_1},\dots,\d_{\eta_\ell}$.  Under this trivialization, $Q$ takes the local form:
	
	\begin{equation}\label{eqn_Q}
	Q=\sum\limits_it_i\d_{\xi_i}+\sum\limits_{i}c_i\xi_i\d_{t_i}+\sum\limits_{i}f_i\d_{s_i}+\sum\limits_{i}g_i\d_{\eta_i}.
	\end{equation}
	where $f_1,\dots,f_r,g_1,\dots,g_{\ell}$ are some functions.  Further, by our choice of coordinates we have
	\[
	Q^2=\sum_{i=1}^n c_i(t_i\d_{t_i}+\xi_i\d_{\xi_i})+\sum\limits_id_is_i\d_{s_i}+\sum\limits_ie_i\eta_i\d_{\eta_i}.
	\]
	where $d_i,e_i\in \Bbbk$.  Define the following vector field on $U$: 
	\[
	h:=\xi_1\d_{\xi_1}+\dots+\xi_n\d_{\xi_n}.
	\]
	Then $h$ commutes with $Q^2$ and has integral eigenvalues bounded between 0 and $n$.  We may decompose $Q$ as 
	\[
	Q=Q_{-1}+Q_0+Q_1+\dots+Q_{n}
	\]
	according to the eigenvalues of $h$. By \cref{eqn_Q} we have  
	\[
	Q_{-1}=t_1\d_{\xi_1}+\dots+t_k\d_{\xi_k}.
	\]
	
	\textbf{Realizing $Q_{-1}$ as defining a Koszul complex:}  Write $A=\Bbbk[U]$; then since $U$ is affine and smooth, it is split.  By Lemma \ref{kers define splitting}, we may split $A$ as 
	\[
	A=A_0\otimes{\bigwedge}^\bullet\langle\xi_1,\dots,\xi_n,\eta_1,\dots,\eta_\ell\rangle.
	\]
	where $A_0=\ker\d_{\xi_1}\cap\cdots\cap\ker\d_{\xi_n}\cap \ker\d_{\eta_1}\cap\cdots\cap\ker\d_{\eta_\ell}$.  In particular $t_1,\dots,t_n\in A_0$.  Then we have
	\[
	I_{Y\cap U}=\left((t_1,\dots,t_n)_0\oplus A_0\otimes{\bigwedge}^+\langle\xi_1,\dots,\xi_k\rangle\right)\otimes{\bigwedge}^\bullet\langle\eta_1,\dots,\eta_\ell\rangle.
	\]
	Here $(t_1,\dots,t_n)_0$ denotes the ideal generated by $t_1,\dots,t_n$ in $A_0$, which is exactly the ideal defining $Y_0\cap U_0$ in $U_0$.  Now one sees that $Q_{-1}$ is defining the Koszul complex on 
	\[
	A_0\otimes{\bigwedge}^\bullet\langle\xi_1,\dots,\xi_n\rangle
	\]
	and the cohomology is $A_0/(t_1,\dots,t_n)_0$ because $t_1,\dots,t_n$ define a regular sequence in $A_0$.  Thus the cohomology of $Q_{-1}$ on
	\[
	(t_1,\dots,t_n)_0\oplus A_0\otimes{\wedge}^+\langle\xi_1,\dots,\xi_k\rangle
	\]
	is trivial, implying that its cohomology on $I_{Y\cap U}$ is also trivial, i.e. $DS_{Q_{-1}}I_{Y\cap U}=0$, since $Q_{-1}(\eta_i)=0$.  
	
	\textbf{Obtaining that $DS_{Q}I_{Y\cap U}=0$:} To show that $DS_{Q}I_{Y\cap U}=0$, we appeal to the following lemma.

	\begin{lemma}\label{lemma_technical}
		Suppose that $h,Q_{-1},Q_0,\dots,Q_{n}$ are operators on a super vector space $V$ such that
		\begin{enumerate}
			\item  $h$ is even and semisimple with bounded, integral eigenvalues;
			\item each $Q_{i}$ is an odd operator;
			\item 	$[h,Q_i]=iQ_i$.
		\end{enumerate}
		If $DS_{Q_{-1}}V=0$, then $DS_{Q}V=0$, where $Q=Q_{-1}+Q_0+\dots+Q_{n}$.
	\end{lemma}

	\begin{proof}
	Let $v\in V$ be such that $Qv=0$, and write $v=v_{i}+v_{i+1}+\dots+v_{j}$ for the $h$-eigenspace decomposition of $v$, where we assume that $i$ is such that $v_i\neq0$.  Then clearly $Q_{-1}v_i=0$, so by assumption there exists $w=w_{i+1}+\dots$ such that $v-Qw$ only has nonzero components in eigenspaces of $h$ with eigenvalue strictly larger than $i$.  Because $h$ has bounded eigenvalues on $V$, we may thus continue inductively to construct $v'\in V$ such that $Qv'=v$, completing the proof.
	\end{proof}

	\textbf{Completing the proof for $\VV=\OO_X$:} Now if $V\sub U$ is an affine $\q$-stable open subvariety of $U$, then we will still have our coordinates above and so the same arguments imply that $DS_QI_{Y\cap V}=0$, which completes the proof in the case when $\VV=\OO_X$.
	
	\textbf{The case of $\VV$ general:} Now we need to deal with the case of general $\VV$.  However we use hypothesis (4) and apply Lemma \ref{fixed_point_neighborhood} to the case of $X=U_z\cap U$ and $\WW=\VV$ to find a $\q$-stable, affine open neighborhood $U'$ of $z$ on which $\VV|_{U'}\cong\OO_{X}\otimes\VV|_{z}$.   Now on $U'$ we have $\q$-equivariant isomorphisms
	\[
	(\VV_Y)|_{U'}=(\II_Y\otimes_{\OO_X}\VV)|_{U'}\cong (\II_Y|_{U'})\otimes_{\Bbbk}\VV|_z.
	\] 
	In particular we have a $\q$-module isomorphism
	\[
	\Gamma(U',\VV_Y)=\Gamma(U',\II_Y)\otimes_{\Bbbk}\VV|_z.
	\]
	Since $DS_Q\Gamma(V,\II_Y)=0$ for any affine open subvariety $V\sub U'$, we have finished the proof of \cref{main_prop}, and thus also of \cref{main_thm}.
\end{proof}

	\section{The supergroup \texorpdfstring{$\widetilde{G_u}$}{~Gu}}
	
	In this section, $G$ will denote an affine algebraic supergroup over $\Bbbk$.  Then $\Bbbk[G]$ is a Hopf superalgebra, and we write $\Delta$ for its coproduct map.  Recall that $G_0\sub G$ will be an affine algebraic group.
	
	By the Lie superalgebra $\g:=\operatorname{Lie}G$, we mean the Lie algebra of left-invariant vector fields on $G$.  We write $u$ for an element of $\g$, and we have a natural isomorphism $\g\to T_eG$, $u\mapsto u_e$, given by evaluation of the vector field at the identity. 
	
	The supergroup $G$ acts on itself by conjugation, inducing an action of $\g$ on $G$ by vector fields.  For $u\in\g$, the vector field induced by this action is:
	\[
	u_{ad}=u_R+u_L=(1\otimes u_e-u_e\otimes 1)\circ\Delta.
	\] 
	where $u_R=(1\otimes u_e)\circ\Delta$ and $u_L=-(u_e\otimes 1)\circ\Delta$ are the infinitesimal right and left translations in the $u_e$ direction.  We refer to $u_{ad}$ as the adjoint vector field of $u$.
	
	This adjoint action of $\g$ on $\Bbbk[G]$ respects the Hopf superalgebra structure, i.e. all morphisms coming from the Hopf superalgebra structure are morphisms of $\g$-modules.  Let $u\in\g_{\ol{1}}$ be such that $c=u^2$ has a semisimple adjoint vector field.  Then since $DS_u$ is a tensor functor, $\Bbbk[G]_u:=DS_{u_{ad}}\Bbbk[G]$ will be a supercommutative Hopf superalgebra.  
	\begin{definition}
		We write $\widetilde{G_u}$ for the algebraic supergroup with $\Bbbk[\widetilde{G_u}]=\Bbbk[G]_u$, and write $\widetilde{\g_u}$ for its Lie superalgebra.
	\end{definition}
	
	If $V$ is a $G$-module, then we have a morphism $V\to V\otimes\Bbbk[G]$, so when we take $DS_u$ we obtain a morphism $V_u\to V_u\otimes \Bbbk[G]_u=V_u\otimes \Bbbk[\widetilde{G_u}]$, which gives $V_u:=DS_uV$ the natural structure of a $\widetilde{G_u}$-module; in this way $\widetilde{G_u}$ acts on the functor $DS_u$. 
	
	The supergroup $\widetilde{G_u}$ was introduced in \cite{sherman2022symmetriesDS}, and was computed in several cases when $G=GL(m|n),Q(n)$, or when $G$ is 'split', meaning that $\g_{\ol{1}}$ is an ideal of $\g$ (equivalently $G=G_0\ltimes\Pi V$, where $V$ is a finite-dimensional $G_0$-module).  However the computations heavily relied on being able to work in coordinates and perform explicit computations.  We will see that with the localization theorem we can compute the supergroup $\widetilde{G_u}$ for a wide range of $u$, and without appealing to coordinates.
	
	In what follows, we write $C(u)$ for the centralizer subgroup of $u$ in $G$.  We begin with the following:
	
	\begin{lemma}\label{lemma_vanishing_adjoint}
		The vanishing set of the adjoint vector field $u_{ad}$ is $C(u)_0\sub G_0$.
	\end{lemma}

	\begin{proof}
		Let $g\in G_0(\Bbbk)$; then by left/right translation invariance, we have $(u_R)(g)=L_{g^{-1}}^*(u_e)$, and $(u_L)(g)=R_{g^{-1}}^*(-u_e)$, where $L_{g^{-1}}$, resp. $R_{g^{-1}}$ denote the left, resp. right translation by $g^{-1}$.  Therefore $u_{ad}(g)=0$ exactly if
			\[
			L_{g^{-1}}^*(u_e)+R_{g^{-1}}^*(-u_e)=0\iff \Ad(g)(u_e)=u_e.
			\]
	\end{proof}
	
	\begin{prop}\label{prop to compute G_u} 
		Let $u\in\g_{\ol{1}}$ be such that $c=u^2$ has that $c_{ad}$ is semisimple.  Let $C(c)\sub G$ be the centralizer of $c$ in $G$.  Then the map $\Bbbk[G]\to\Bbbk[C(c)]$ induces a natural isomorphism of Hopf superalgebras
		\[
		DS_u\Bbbk[G]\cong DS_u\Bbbk[C(c)],
		\]
		i.e. a natural isomorphism of supergroups $\widetilde{C(c)_u}\cong\widetilde{G_u}$.
	\end{prop}
	\begin{proof}
	This follows from \cref{cor_assume_sq_zero}, using that $C(c)=\Spec\Bbbk[G]/(\operatorname{Im} c)$.
	\end{proof}
	
	\subsection{Computation of \texorpdfstring{$\widetilde{G_u}$}{~Gu} for Kac-Moody supergroups, \texorpdfstring{$Q(n)$}{Q(n)}, and \texorpdfstring{$P(n)$}{P(n)}}  We use Proposition \ref{prop to compute G_u} to compute $\widetilde{G_u}$ for several important supergroups $G$.  In what follows, by a Kac-Moody supergroup we will mean one of the supergroups $GL(m|n)$, $SOSp(m|2n)$, $D(2,1;\alpha)$, $AG(2,1)$, or $AB(3,1)$, where for $D(2,1;\alpha)$, $AG(2,1)$, and $AB(3,1)$ we take the simply connected form of these supergroups. We write $Q(n)$ for the queer supergroup, and $P(n)$ for the periplectic supergroup, both of whom have $G_0\cong GL(n)$.  We refer to \cite{CW} for the Lie superalgebras of these supergroups listed here; however we note that in \cite{CW} the Lie superalgebra of $AG(2,1)$ is written as $G(3)$, while the Lie superalgebra of $AB(3,1)$ is written as $F(3|1)$.
	
	For the Kac-Moody supergroups and for $P(n)$, choose a maximal torus, and let $\s\l(1|1)^k\sub\g$ be a subalgebra corresponding to a choice of odd roots $\alpha_{1},\dots,\alpha_k$ such that $-\alpha_{1},\dots,-\alpha_k$ are also roots, and $\pm\alpha_{i}\pm\alpha_{j}$ is not a root whenever $i\neq j$.  Now choose $u\in\s\l(1|1)^k$ generic in the sense that $c=u^2=(c_1,\dots,c_k)$ is generic, meaning that $c_i\neq0$ for all $i$ and so that we have, apart from the case of $D(2,1;\alpha)$,
	\[
	C(c)=G'\times GL(1|1)^k,
	\]
	where $G'$ is the group in the first factor of the below list.  Such an element $u$ will be said to have rank $k$.  
	
	For $D(2,1;\alpha)$, choose an element $u$ of nonzero square in $\s\l(1|1)$, and let $c=u^2$.  Then $C(c)$ will be a supergroup with Lie superalgebra $\s\l(1|1)+\mathfrak{t}$, where $\mathfrak{t}$ is a maximal torus; it is computed explicitly in the appendix.
	
	For $G=Q(n)$, we choose a maximal torus; then its centralizer in $G$ is given by $Q(1)^n$.  Choose a factor subgroup $Q(1)^k$, and let $u\in\q(1)_{\ol{1}}^k$ be generic, in the sense that if $u=(u_1,\dots,u_k)$ then $u_i\neq0$ for all $i$, and we have:
	\[
	C(c)=Q(n-k)\times Q(1)^{k}.
	\]
	Again, such an element $u$ will be said to have rank $k$.
	
	Before giving the following theorem we recall that $\G_m$ denotes the one-dimensional complex algebraic torus, and $\G_a^{0|n}$ denotes the purely odd abelian supergroup of dimension $(0|n)$.

	\renewcommand{\arraystretch}{1.5}
	
	\begin{thm}\label{thm G_u comp}
		For $u\in\g_{\ol{1}}$ generic of rank $k$ as described above, we have the following table describing $\widetilde{G_u}$:
		\[
		\begin{array}{|c|c|}
			\hline
			G & \widetilde{G_u}\\ 
			\hline
			GL(m|n) & GL(m-k|n-k)\times\G_a^{0|k}\\
			\hline
			SOSp(m|2n) & SOSp(m-2k|2n-2k)\times\G_a^{0|k}\\
			\hline
			D(2,1;\alpha), \alpha\in\Q & \G_m\times\G_a^{0|1}\\
			\hline
			D(2,1;\alpha), \alpha\notin\Q & \G_a^{0|1}\\
			\hline
			G(2,1) & SL(2)\times\G_a^{0|1}\\
			\hline
			AB(3,1) & SL(3)\times\G_a^{0|1}\\
			\hline
			Q(n) & Q(n-k)\times\G_a^{0|k}\\ 
			\hline
			P(n) & P(n-2k)\times\G_a^{0|k}\\
			\hline
		\end{array}
		\]
	\end{thm}
	
	\begin{proof}
		For $D(2,1;\alpha)$, the explicit computation is done in the appendix.  For the rest of the cases when $G\neq Q(n)$, we have by \cref{prop to compute G_u} that
		\[
		DS_u\Bbbk[G]\cong DS_u\Bbbk[C(c)]=DS_u\Bbbk[G'\times GL(1|1)^k]=\Bbbk[G']\otimes DS_u\Bbbk[GL(1|1)^k].
		\]
		Now the computation follows from the fact that for any nonzero $v\in\g\l(1|1)_{\ol{1}}^{hom}$, we have $\widetilde{GL(1|1)_v}=\G_{a}^{0|1}$, and this was proven in Thm.~4.3, \cite{sherman2022symmetriesDS}.  We will also provide another proof of this statement in \cref{prop G u in special cases}.
		
		When $G=Q(n)$ we use the same ideas, only now we use that $\widetilde{Q(1)_v}=\G_{a}^{0|1}$ for nonzero $v\in\q(1)_{\ol{1}}$, which was proven in Thm.~4.3, \cite{sherman2022symmetriesDS} and will also be shown in \cref{prop G u in special cases}.
	\end{proof}
	
	\subsubsection{Realization of the Lie superalgebra}
	
	Let $G$ and $u$ be as in the setup of Theorem \ref{thm G_u comp}, except we exclude now the case $G=Q(n)$.  We give an explicit realization of $\operatorname{Lie}\widetilde{G_u}$ as coming from $\g$, according to the natural construction introduced in Sec.~3 of \cite{sherman2022symmetriesDS}.  
	
	In particular let us assume that $u\in\s\l(1|1)^k$ takes the form
	\[
	u=e_{\alpha_1}+\dots+e_{\alpha_k}+c_1e_{-\alpha_1}+\dots+c_ke_{-\alpha_k},
	\]
	where $c_1,\dots,c_k\in\Bbbk^\times$ are generic and $e_{\pm\alpha_i}\in\g_{\pm\alpha_i}$.  Let $\h'=\sum\limits_{i}[\g_{\alpha_i},\g_{-\alpha_i}]\sub\g_{\ol{0}}$; then $\h'\sub[u,\g]$ and is a toral subalgebra.  By Lemma 3.1 of \cite{sherman2022symmetriesDS}, if $V$ is any representation on which $\h'$ acts semisimply, we have a canonical isomorphism $DS_uV\cong DS_uV^{\h'}$.  Therefore, if we define the Lie superalgebra
	\[
	\g(u,\h'):=DS_u\left(\c(\h')/\h'\right),
	\]
	where $\c(\h')=\g^{\h'}$ is the centralizer of $\h'$ in $\g$, we will obtain a natural action of $\g(u,\h')$ on $DS_uV^{\h'}$.  By \cite{sherman2022symmetriesDS}, we have
	\[
	\g(u,\h')\cong\g'\times\Bbbk\langle e_{-\alpha_1},\dots,e_{-\alpha_k}\rangle,		
	\]
	where $\g'=\operatorname{Lie}G'$, and $G'$ is the corresponding first factor in our table from Theorem \ref{thm G_u comp}. 	 Now $\g(u,\h')$ will have a natural map to $\operatorname{Lie}\widetilde{G_u}$, as follows.  
	
	We have a natural action of $G\times G$ on $G$ by left and right translation, and thus a map $\g\times\g\to \operatorname{Vect}(G)$.  Recall that $u_{ad}$ is the image of $(u,u)\in\g\times\g$ under this map.  We see then that $\h'\times\h'\sub[(u,u),\g\times\g]$, and thus we obtain an action of 
	\[
	DS_{(u,u)}\left(\c(\h'\times\h')/(\h'\times\h')\right)=\g(u,\h')\times\g(u,\h')
	\]
	on $\Bbbk[G]$.  Further, given $(v,0)\in\g(u,\h')\times\g(u,\h')$, the action on $\Bbbk[G]_u$ is given exactly by $(v\otimes 1)\circ\Delta$, where we are using that the coproduct on $\Bbbk[\widetilde{G_u}]=\Bbbk[G]_u$ is induced by $\Delta$.  From this formula it is clear that $\g(u,\h')$ will define right-invariant vector fields on $\widetilde{G_u}$, which gives our desired map $\g(u,\h')\to\operatorname{Lie}\widetilde{G_u}$.
	
	\begin{thm}
		For $G\neq D(2,1;\alpha)$, the map $\g(u,\h')\to \operatorname{Lie}\widetilde{G_u}$ is an isomorphism.
	\end{thm}
	
	\begin{proof}
		We trace through the constructions of each space.  First of all, by Proposition \ref{prop to compute G_u}, we have an isomorphism of Hopf superalgebras
		\[
		DS_u\Bbbk[C(c)]\cong\Bbbk[\widetilde{G_u}].
		\]
		Here $c$ is generic in the sense that it satisfies $C(c)=C(\h')$.  Clearly $\c(\h')=\c(c)$ is the Lie superalgebra of $C(c)$, and now $C(c)=G'\times GL(1|1)^k$, where $u\in\g\l(1|1)^{k}$ is as written above.  Thus it suffices to prove the statement in the case when $\g=\g\l(1|1)$ and $u=e_{\alpha}+ce_{-\alpha}$ for $c\in\Bbbk$; this is done in 4.3.1 of \cite{sherman2022symmetriesDS}.
	\end{proof}
	
	\subsubsection{Split supergroups} We present below a special case of the general result on split supergroups given in Sec.~4.6 of \cite{sherman2022symmetriesDS}.  By a split supergroup $G$ we mean a supergroup for which $[\g_{\ol{1}},\g_{\ol{1}}]=0$.
	\begin{prop}
		Suppose that $G$ is a split supergroup, and $u\in\g_{\ol{1}}$ has a reductive stabilizer $C(u)_0$ in $G_0$.  Then $\widetilde{G_u}$ is the supergroup with $(\widetilde{G_u})_{0}=C(u)_0$ and Lie superalgebra $\c(u)/[u,\g_{\ol{0}}]$.
	\end{prop}
	\begin{proof}
		Since $C(u)_0$ is reductive and $[u,\g_{\ol{0}}]$ is a $C(u)_0$-module, we may find a complimentary $C(u)_0$-module $V$ so that $\g_{\ol{1}}=[u,\g_{\ol{0}}]\oplus V$.  In particular, we have the subgroup $K$ of $G$ with $K_0=C(u)_0$ and Lie superalgebra $\k=\c(u)_{\ol{0}}\oplus V$.  We apply now Theorem \ref{main_thm} with $Y=K$ and $Z=Y_0$; notice that to check condition (3), we need only look at the identity $e\in K(\Bbbk)$ since we may translate by $C(u)_0$.  Now our map on the fiber of the conormal bundle at $e$ is exactly dual to $[u,-]:\g_{\ol{0}}/\c(u)_{\ol{0}}\to [u,\g_{\ol{0}}]$, which is clearly an isomorphism.  Thus we have
		\[
		DS_u\Bbbk[G]=DS_u\Bbbk[K]=\Bbbk[K].
		\]

	\end{proof}
	
	\subsection{Identifying the odd tangent bundle}
	We are able to compute several more supergroups $\widetilde{G_u}$, but in order to do so we need another tool.  For an affine variety $X_0$ we define the supervariety $\Pi\TT_{X_0}$, the odd tangent bundle of $X_0$, to have functions given by $\Omega^\bullet_{X_0}$, the algebra of de Rham differentials on $X_0$, where $\Omega_{X_0}^{1}$ consists of odd elements.  The supervariety $\Pi\TT_{X_0}$ admits a canonical odd, everywhere vanishing vector field $d$ given by the de Rham differential.  Observe further that if $U_0$ is an affine open subvariety of $X_0$, then the open subvariety it determines of $\Pi\TT_{X_0}$ is exactly $\Pi\TT_{U_0}$.
	
	\begin{lemma}\label{lemma odd cotang}
		Let $X$ be a smooth affine supervariety, and let $u$ be an odd, everywhere vanishing vector field on $X$ such that for all $x\in X(\Bbbk)$ we have:
		\begin{enumerate}
			\item $u(T_x^*X)_{\ol{1}}=0$;
			\item the map $u:(T^*_xX)_{\ol{0}}\to(T_x^*X)_{\ol{1}}$ is an isomorphism.
		\end{enumerate}
  Then there exists an isomorphism $X\cong \Pi\TT_{X_0}$ such that $u$ corresponds to the de Rham differential.
	\end{lemma}
	
	\begin{proof}
		We define an isomorphism $X\to \Pi\TT_{X_0}$ locally and prove that it glues.  
		
		First of all, because $X$ is smooth and affine, it is split by \cite{voronov1990elements}, so we can split $\Bbbk[X]\to\Bbbk[X_0]$, allowing us to identify (non-canonically) $\Bbbk[X_0]$ as a subalgebra of $\Bbbk[X]$.  Thus we do so, i.e. we choose once and for all such a subalgebra $\Bbbk[X_0]\sub\Bbbk[X]$. 
		
		Suppose that $U_0$ is an open subset of $X_0$ on which $X_0$ has coordinates $x_1,\dots,x_n$.  Write $U$ for the corresponding open subset of $X$.  Let $\xi_i:=u(x_i)\in\Bbbk[U]$; then by our assumption, $\xi_1,\dots,\xi_n$ define a full set of odd coordinates on $U$.  
		
		We now define an isomorphism $\varphi_U:U\to\Pi\TT_{U_0}$ on functions by mapping $\Bbbk[U_0]={\bigwedge}^0\Omega_{X}\sub\Bbbk[\Pi\TT_{U_0}]$ identically to our copy of $\Bbbk[U_0]$ determined by our global splitting, and sending $dx_i$ to $\xi_i$.  It is clear that under this local isomorphism $u$ corresponds to the de Rham differential $d$; in particular, on $U$ we have $u=\sum\limits\xi_i\d_{x_i}$.
		
		Now we need to check these maps glue.  However this follows from showing that the above isomorphism is independent of the coordinates we chose.  Thus suppose on $U_0$ we have another set of coordinates $y_1,\dots,y_n$.  Then if we set $\eta_j:=u(y_j)$, we have 
		\[
		\eta_j=u(y_j)=\sum\limits_{i}\xi_i\d_{x_i}(y_j)=\sum\limits_{i}\varphi_U^*(dx_i\d_{x_i}(y_j))=\varphi_U^*(dy_j).
		\]
		It follows that $\varphi_U$ is indeed independent of the choice of coordinates; it is only dependent on our choice of splitting $\Bbbk[X_0]\sub\Bbbk[X]$, but this was global and fixed from the start.
	\end{proof}

	\begin{remark}
	Note that \cref{lemma odd cotang} was proven in \cite{vaintrob1996normal} in greater generality, although they work in the smooth setting and use different tools. 
	\end{remark}

	With Lemma \ref{lemma odd cotang} in hand we now prove several fun results.  For context, we recall that it was originally Hopf who observed in \cite{hopf1964topologie} that the de Rham cohomology, $H_{dR}^\bullet(G)$, of a Lie group $G$ admits the structure of a graded commutative, cocommutative  Hopf algebra, and therefore has spectrum given by an odd abelian supergroup. 
	
	\begin{prop}\label{prop special Q case}
		Let $G=Q(n)$, and let
		\[
		u=\begin{bmatrix}
			0 & I_n\\ I_n & 0
		\end{bmatrix}
		\]
		Then $\widetilde{G_u}\cong\G_{a}^{0|n}$.
	\end{prop}
	
	\begin{proof}
		In this case the adjoint vector field of $u$ vanishes everywhere, and we have $[u,\g_{\ol{0}}]=0$ and $[u,-]:\g_{\ol{1}}\to\g_{\ol{0}}$ is an isomorphism.  Thus Lemma \ref{lemma odd cotang} applies, and we obtain that
		\[
		\Bbbk[Q(n)]_u=H_{dR}^\bullet(GL(n))\cong\G_{a}^{0|n}.
		\]
	\end{proof}
	
	\begin{prop}\label{prop G u in special cases}
		\begin{enumerate}
			\item Let $G=GL(n|n)$, and let
			\[
			u=\begin{bmatrix}
				0 & I_{n}\\ \lambda I_{n} & 0
			\end{bmatrix}
			\]
			in block form, where $\lambda\in\Bbbk$.  Then 
			\[
			\widetilde{GL(n|n)_u}\cong\G_{a}^{0|n}.
			\]
			\item Let $G=P(n)$, and let 
			\[
			u=\begin{bmatrix}
				0 & A\\ 0 & 0
			\end{bmatrix}
			\]
			in block form, where $A$ is any invertible symmetric matrix.  Then
			\[
			\widetilde{P(n)_u}\cong\mu_2\times\G_{a}^{0|\lfloor n/2\rfloor}.
			\]
			\item Let $G=P(2n)$, and let
			\[
			u=\begin{bmatrix}
				0 & 0 \\ B & 0
			\end{bmatrix}
			\]
			in block form, where $B$ is any invertible skew-symmetric matrix.  Then
			\[
			\widetilde{P(2n)_u}\cong\G_{a}^{0|n}.
			\]
		\end{enumerate}
	\end{prop}
	
	\begin{proof}
		The proofs of (1), (2), and (3) are identical.  Namely in each case we have a decomposition of our Lie superalgebra $\g=\g_{-1}\oplus\g_0\oplus\g_1$ where we can assume for (2) and (3) that $u\in\g_1$, and for (1) we write $u=u_++\lambda u_-$, where $u_{\pm}\in \g_{\pm1}$.  
		
		Let $K\sub G_0$ be the stabilizer of $u$ in $G_0$; then we  have in each case that $[u,\g_{-1}]=\k:=\operatorname{Lie}K$.  Further in all cases $u$ defines an isomorphism:
		\[
		[u,-]:(\g_0\oplus\g_{-1})/(\k\oplus\g_{-1})\to \g_{1}.
		\]
		Let $Y\sub G$ be the subgroup with $Y_0=K$ and with Lie superalgebra $\k\oplus\g_{-1}$.  Then the hypotheses of Theorem \ref{main_thm} are satisfied with $Z=Y_0$, and so we learn that
		\[
		DS_u\Bbbk[G]=DS_u\Bbbk[Y].
		\]
		However on $Y$ we have that $u$ is everywhere vanishing and satisfies all the hypotheses of Lemma \ref{lemma odd cotang}, so that we obtain
		\[
		DS_u\Bbbk[Y]\cong H_{dR}^\bullet(K).
		\]
		Thus to finish the proof we see that for (1) we have $K=GL(n)$, for $(2)$ we have $K=O(n)$, and (3) we have $K=Sp(2n)$.
	\end{proof}
	
	\begin{remark}
		We may realize the Lie superalgebra of $\widetilde{G_u}$ for the three cases presented in \cref{prop G u in special cases} according to the construction described in Section 5 of \cite{sherman2022symmetriesDS}.  Namely, if we set $\UU(\g,\k):=(\UU\g)^\k/(\k\UU\g+\UU\g\k)^\k$, then $\operatorname{ad}(u)$ preserves $\UU(\g,\k)$ and we obtain a natural map
		\[
		\UU(\g,\k)_u\to\UU\operatorname{Lie}\widetilde{G_u}.
		\] 
		Then the map $\UU\g_{-1}^{\k}\to\UU(\g,\k)$ lands in $\ker(u)$, and in fact maps injectively into $\UU(\g,\k)_u$, inducing an isomorphism:
		\[
		\UU\g_{-1}^{\k}\to\UU\operatorname{Lie}\widetilde{G_u}.
		\]
	\end{remark}
	
	\subsection{Determination of \texorpdfstring{$\widetilde{G_u}$}{~Gu} for all \texorpdfstring{$u$}{u} when \texorpdfstring{$G=GL(m|n)$}{G=GL(m|n)} or \texorpdfstring{$Q(n)$}{Q(n)}}  In the following corollary, we use the notion of rank introduced in Section 3 of \cite{sherman2022symmetriesDS}.
	\begin{cor}
		\begin{enumerate}
			\item 	Let $u\in\g\l(m|n)_{\ol{1}}$ with $u^2$ semisimple, and such that $u$ is of rank $r$.  Then
			\[
			\widetilde{GL(m|n)_u}\cong GL(m-r|n-r)\times \G_{a}^{0|r}.
			\]
			\item Let $u\in\q(n)_{\ol{1}}^{hom}$ with $u^2$ semisimple, and such that $u$ is of rank $r$ and \newline $u=\begin{bmatrix}
				0 & B\\ B &0\end{bmatrix}$, where $B$ is of rank $s$.  Then
			\[
			\widetilde{Q(n)_u}\cong Q(n-r)\times \G_{a}^{0|s}.
			\]
		\end{enumerate}
	\end{cor}
	\begin{proof}
		For both of the cases above, we apply Proposition \ref{prop to compute G_u}.  In the $GL$ case, $C(c)$ will be a product of $GL$'s of smaller rank, say $G_1\times\cdots\times G_\ell$, where $G_i=GL(a_i|b_i)$; then we can decompose $u=u_1+\dots+u_{\ell}$, and in each case $u_i\in\g_{i}$ will have that $u_i^2$ is central in $\g_i$.  Therefore either $u_i^2=0$ and we can conclude from the results of Thm.~4.3 of \cite{sherman2022symmetriesDS}, or $a_i=b_i$ and $u_i$ is conjugate to the matrix given in (1) of Proposition \ref{prop G u in special cases}.
		
		For the $Q$ case, $C(c)$ will be a product of $Q$'s of smaller rank, $Q(b_1)\times\cdots\times Q(b_\ell)$, and we can again decompose $u=u_1+\dots+u_{\ell}$.  Then once again either $u_i\in\q(b_i)$ will satisfy $u_i^2=0$ allowing us to conclude from Thm.~4.3 of \cite{sherman2022symmetriesDS}, or it will be a multiple of the matrix in Proposition \ref{prop special Q case}.
	\end{proof}

	\section{Homogeneous superspaces}
	
	Recall that a supergroup $G$ is called quasireductive if $G_0$ is reductive.  We will always assume in this section that $G$ is quasireductive, and write $\Rep G$ for the category of $G$-modules.  We refer to \cite{Ser2} for more on the theory of such supergroups.  Write $\g$ for its Lie superalgebra, which will be quasireductive, and set
	\[
	\g_{\ol{1}}^{hom}:=\{u\in\g_{\ol{1}}:u^2\text{ is semisimple}\}.
	\]
	An element $u\in\g_{\ol{1}}^{hom}$ has that $u^2$ acts semisimply on any $G$-module, and thus the functor $DS_u$ is defined on $\Rep G$.
	
    We will consider homogeneous spaces $G/K$ for $K$ a quasireductive subgroup of $G$; in this case $G/K$ is a smooth affine supervariety.  See \cite{MaT} for more on the construction of this quotient and proof of affinity of $G/K$. 
	
	The Lie superalgebra $\g$ acts by vector fields on $G/K$ via infinitesimal left translation; for an element $v\in\g$ we will, by abuse of notation, write $v$ for the vector field it defines on $G/K$.  We continue to write $Z(v)$ for the vanishing set of $v$ on $G/K$.

	\begin{lemma}
		For $v\in\g$, and $g\in G(\Bbbk)$, $v$ vanishes at the point $gK$ on $G/K$ if and only if $\Ad(g^{-1})(v)\in\k$.  
	\end{lemma}

	\begin{proof}
		Indeed, the stabilizer of $gK$ is $gKg^{-1}$.
	\end{proof}

	For a homogeneous element $v\in\g$, we write $N_G(v)$ for the normalizer of $v$ in $G$, and $C_G(v)$ for the centralizer of $v$ in $G$.  Notice that for a homogeneous space $G/K$, if $v\in\k$ then $N_{G_0}(v)\cdot v\sub\k$, and thus we have
	\[
	N_{G_0}(v)K\sub Z(v).
	\]	
	\begin{prop}\label{prop homog spaces}
		Let $G$ be quasireductive and $K\sub G$ a closed quasireductive subgroup.  Let $u\in\k\cap \g_{\ol{1}}^{hom}$ with $c=u^2$, and suppose that
		\[
		Z(u)\sub N_{G_0}(c)K.
		\]
		Then 
		\[
		DS_u\Bbbk[G/K]=DS_u\Bbbk[N_G(c)/(N_G(c)\cap K)].
		\]
	\end{prop}
	
	\begin{proof}
	Let $Y=N_G(c)/(N_G(c)\cap K)$, which is a smooth, closed subvariety of $G/K$.  Then because $\operatorname{ad}(c)$ acts with non-zero eigenvalues on $\g/\operatorname{Lie}N_G(c)$, it will act by an automorphism on $\NN^\vee_Y|_{y}$ for any $y\in Y(\Bbbk)$, since all such $y$ have representatives in $N_G(c)$.  By our assumption, we have that $Z(u)\sub Y_0$, so $u$ will act by an automorphism on $\NN^\vee_Y|_{z}$ for any $z\in Z(u)$, and thus the hypotheses of Theorem \ref{main_thm} hold. 
	\end{proof}

	\begin{remark}
		We remind for future use that if $G_0$ is a reductive algebraic group (in particular it is connected) and $c$ is a semisimple element of its Lie algebra, then $C(c)$, the centralizer of $c$ in $G_0$, is connected.
	\end{remark}
	
	Before discussing symmetric spaces, we give an example in which condition (4) of Theorem \ref{main_thm} holds for any finite-dimensional equivariant vector bundle.	
	
	\begin{example}
		Consider $G=GL(n|n)$, $K=GL(1|1)\times\cdots\times GL(1|1)$.  Let $u\in\k_{\ol{1}}$ be generic, i.e. so that if $c:=u^2$ then $N_G(c)=K$.  Then we claim that $Z(u)=S_n\cdot K$, where $S_n\sub GL(n|n)_0\cong GL(n)\times GL(n)$ consists of the diagonally embedded permutation matrices.  Indeed, it is clear that this subgroup $S_n$ stabilizes $K$, so that $S_n\cdot K\sub Z(u)$. On the other hand if $g\in Z(u)$, then necessarily $\Ad(g)(c)\in\k_{\ol{0}}$, which implies, by regularity properties of $c$, that $g$ lies in the Weyl group $S_n\times S_n$ up to $K_0$.  It is straightforward to check that only the diagonally embedded copy of $S_n$ in $S_n\times S_n$ will keep $u$ inside $\k_{\ol{1}}$, so this proves our claim.
		
		Now, we see that the hypotheses of \cref{prop homog spaces} do not apply, since $N_{G_0}(c)=K_0$ in this case; however we can still compute $DS_u\Bbbk[G/K]$ as follows.
		
		Let $Y=Z(u)=S_n\cdot K$; i.e. $Y$ is a purely even variety consisting of the $n!$ points $\sigma K$ for $\sigma\in S_n$ (diagonally embedded in $GL(n)\times GL(n)\cong G_0$).  Thus for $\sigma\in S_n$, $\NN^\vee_{Y}|_{\sigma K}=T_{\sigma K}^*(G/K)$; as a $\k$-module, this cotangent space is isomorphic to $(\g/\k)^{\sigma}$, i.e. the $\sigma$ twist of the $\k$-module $\g/\k$.  Thus it is clear that $c$ acts by an automorphism on these fibers, and therefore $u$ does as well.  Therefore the hypotheses of \cref{main_thm} apply, giving:
		\[
		DS_u\Bbbk[G/K]=\prod\limits_{\sigma\in S_n}\Bbbk.
		\]
		Now let $V$ be a $K$-module, and write $\VV:=G\times_KV$, i.e. the $G$-equivariant vector bundle on $G/K$ determined by $V$.  We may write $\g=\k\oplus\m$ for a $\k$-stable subspace $\m$, and consider the open subscheme $W\sub\m$ defined, via the functor of points, as 
		\[
		W(R):=\{A\in\m(R):I_{n|n}+A\text{ is invertible in }\End(R^{n|n})\},
		\]
		where $R$ is a supercommutative algebra.  We see that $W$ is an open, $\Ad(K)$-stable subset of $\m$ containing $0$. We have the map
		\[
		\phi:W\to GL(n|n)/K, \ \ \ \ A\mapsto (I_{n|n}+A)K,
		\]
		where the formula is defined on $R$-points, for a supercommutative algebra $R$.  One may check that $\phi$ is $K$-equivariant, injective on closed points, and has that $d_0\phi$ is an isomorphism.  Thus we may choose a $K$-stable open subvariety $U\sub W\sub \m$, containing $0$, such that $\phi$ restricts to an open embedding $U\hookrightarrow GL(n|n)/K$.  Further, by Sumihiro's Theorem and the fact that $K_0$ is torus, we may assume that $U$ is affine.
		
		Observe now that if we write $\pi:G\to G/K$ for the canonical projection, then we obtain a $K$-equivariant isomorphism $\pi^{-1}(\phi(U))\cong K\times U$, where $K$ acts on the right on $\pi^{-1}(U)$ and $K$, and by the adjoint action on $U$.
		
		It follows that we have $\Gamma(U,\VV)\cong(\Bbbk[K]\otimes\Bbbk[U]\otimes V)^{K}$, which in particular contains the $K$-submodule $(\Bbbk[K]\otimes\Bbbk\langle 1\rangle\otimes V)^{K}\cong V$. This $K$-submodule of $\Gamma(U,\VV)$ defines a local, $K$-equivariant trivialization of $\VV$.  Since $G/K$ is a homogeneous space, we may translate the above trivialization to obtain similar local trivializations around each point $\sigma K$ for $\sigma\in S_n$. It follows that the hypotheses of \cref{main_thm} hold, and so we obtain that
		\[
		DS_u\Gamma(G/K,\VV)\cong \bigoplus\limits_{\sigma\in S_n}DS_uV^{\sigma^{-1}}.
		\]
	\end{example}

	Before moving to more general computations on supersymmetric spaces, we give a specific example which is of interest on its own for the elegance of its answer.
	
	\begin{example}[Supersphere]
		Let $\mathbb{S}^{m-1|2n}=OSp(m|2n)/OSp(m-1|2n)=G/K$ denote the supersphere. Explicitly, $G/K$ is the closed subvariety of $\Bbbk^{m|2n}$ cut out by the equation $q=1$, where $q$ is a nondegenerate, supersymmetric quadratic form preserved by $OSp(m|2n)$ on $\Bbbk^{m|2n}$.  In particular if $m=1$, $(G/K)_0=\mathbb{S}^0$ consists of two points.
		
		Note that if $m>1$ then we have $\mathbb{S}^{m-1|2n}=SOSp(m|2n)/SOSp(m-1|2n)$.  Let us assume that $m\geq 3$ and $n\geq 1$. 
		
		Choose a subgroup $S\cong SL(1|1)\sub K$ corresponding to an odd isotropic root of $K$ (with respect to some maximal torus), and let $u\in\operatorname{Lie}S\cong \s\l(1|1)_{\ol{1}}$ be such that $c=u^2\neq 0$.  Then there exists an odd isotropic root (with respect to some maximal torus) of $G$ with corresponding subgroup $S'\cong SL(1|1)$ such that $u$ is conjugate to an element in $\operatorname{Lie}S'$ in $G$ (note: $S$ arises from an isotropic root of $K$, not $G$, whereas $S'$ arises from an isotropic root of $G$, not $K$).  We show in \cref{sec_symm_spaces} that
		\[
		Z(u)\sub N_G(c)_0K,
		\]
		so we obtain from \cref{prop homog spaces} that
		\[
		DS_u\Bbbk[\mathbb{S}^{m-1|2n}]=DS_u\Bbbk[N_G(c)/(N_G(c)\cap K)].
		\]
		Now our copy of $SL(1|1)$ is normalized by $N_G(c)$, and so since it is contained in $K$ it will act trivially on $N_G(c)/N_G(c)\cap K$, meaning that in particular $u$ does too, so that
		\[
		DS_u\Bbbk[\mathbb{S}^{m-1|2n}]=\Bbbk[N_G(c)/(N_G(c)\cap K)].
		\] 
		For $g\in N_G(c)(\Bbbk)$, we have $\Ad(g)c=\pm c$; if $m>3$ then $K(\Bbbk)$ contains an element $g$ (coming from the Weyl group of $K$) such that $\Ad(g)c=-c$, and we have in this case that
		\[
		N_G(c)/(N_G(c)\cap K)=C_G(c)/(C_G(c)\cap K).
		\]
		Here $C(c)=SOSp(m-2|2n-2)\times GL(1|1)$ and $C(c)\cap K=SOSp(m-3|2n-2)\times GL(1|1)$, so $C_G(c)/(C_G(c)\cap K)=\mathbb{S}^{m-3|2n-2}$.  If $m=3$, then the even part of $N_G(c)/(N_G(c)\cap K)$ has two points, with one point being $C_G(c)/(C_G(c)\cap K)$.  From this we obtain that $N_G(c)/(N_G(c)\cap K)=OSp(1|2n-2)/Sp(2n)=\mathbb{S}^{m-3|2n-2}$. Thus in general we have the beautiful formula:
		\[
		DS_u\Bbbk[\mathbb{S}^{m-1|2n}]=\Bbbk[\mathbb{S}^{m-3|2n-2}].
		\]
	\end{example}

	\subsection{Supersymmetric spaces}\label{sec_symm_spaces} We apply Theorem \ref{main_thm} to several spaces of the form $G/K$ where $K$ is a symmetric subgroup of $G$.  By definition, a symmetric subgroup $K$ is one such that there exists an involution $\theta$ of $G$ with $(G^\theta)^\circ\sub K\sub G^\theta$, where $G^\theta$ is the fixed points of $\theta$, and $(G^\theta)^\circ$ is the connected component of the identity. 
	
	Note that the diagonally embedded subgroup $G\sub G\times G$ is a symmetric subgroup, and we have $G\cong (G\times G)/G$, so in fact our study of $\widetilde{G_u}$ in Section 5 is a special case of these examples.
	
	Let $G$ be one of $GL(m|n)$, $SOSp(m|2n)$, $P(n)$, or the simply connected supergroup of an exceptional basic simple Lie superalgebra.  Then we will consider symmetric subgroups $K$ of $G$ with an element $u\in\k_{\ol{1}}^{hom}$ which is generic of minimal rank, and satisfies the hypothesis of Proposition \ref{prop homog spaces}, namely that
	\[
	Z(u)\sub N_G(c)_0K. \ \ \ \ \ (*)
	\]
	 This condition will hold exactly in the cases for which $\k$ admits at most one factor subalgebra $\k'$ which has that $(\k'_{\ol{1}})^{hom}\neq0$.  One can easily go through the classification of supersymmetric pairs, found in \cite{Ser3}, to check which ones these are.
	 
	Let us explain how we compute in each case.  First, choose a maximal torus $\t'\sub\k_{\ol{0}}$ and extend it to a maximal torus $\t$ of $\g_{\ol{0}}$.  For the cases when $\k\neq\q(n)$, choose a root subalgebra $\s\l(1|1)\sub\k$, and let $u\in\s\l(1|1)_{\ol{1}}$ such that $c=u^2\neq0$.  First suppose that we have
	\[
	\operatorname{Lie}C_G(c)\cong \g\l(1|1)\times\g_u.
	\]
	The symmetric pairs which will satisfy this along with (*) are:
	\[
	(\g\l(m|n),\g\l(m|n-r)\times\g\l(r)), \ \ (\o\s\p(m|2n),\o\s\p(m|2n-2r)\times\s\p(2r)),
	\]
	\[
	(\o\s\p(m|2n),\o\s\p(m-1|2n-2r)\times\o\s\p(1|2r)), \ \ (\o\s\p(m|2n),\o\s\p(m-r|2n)\times\o(r)),
	\]
	\[
	(\o\s\p(2m|2n),\g\l(m|n)),  \ \ (\p(n),\g\l(r|n-r)),  
	\]
	\[
	(\mathfrak{d}(1|2;\alpha),\o\s\p(2|2)\times\s\o(2)), \ \ (\a\b(1|3),\g\o\s\p(2|4)), \ \ (\a\b(1,3),\s\l(1|4)),
	\]
	\[
	(\a\b(1|3),\mathfrak{d}(1|2,2)\times\s\l(2)), \ \ (\g(1|2),\mathfrak{d}(1|2;3)), \ \ (\g(1|2),\o\s\p(3|2)\times\s\l(2)).
	\]
	
	In all of the above cases, an element $c$ with this centralizer (up to isomorphism) is unique in $\t'$ up to: adding a central element in $\g$, acting by the Weyl group of $\t'$ in $K$, and scaling.  Further, either we have $N_G(c)=C_G(c)$ or $N_G(c)/C_G(c)\cong \mu_2$, and the action of $\mu_2$ on $c$ is $c\mapsto-c$.
	
	We now show that (*) holds for the pairs listed above.  Suppose that in one of the listed cases we have $u$ vanishes at $gK$; then in particular $c$ vanishes here, and we have $\Ad(g^{-1})(c)\in\k$.  Choose $k\in K_0$ such that $c'=\Ad(gk)^{-1}(c)\in\t'$.  Clearly we have $\operatorname{Lie}C(c')\cong\g\l(1|1)\times\g_u$.  Now to obtain $c$ from $c'$, the only case where we could add a central element in $\g$ is $\g=\g\l(m|n)$, and by rank considerations we see that this couldn't happen.  Thus there exists $w\in K_0$, corresponding to a lift of a Weyl group element, such that $w^{-1}c'$ is a scalar multiple of $c$. Therefore $gkw\in N_{G}(c)_0$, showing that the condition (*) indeed holds.  Therefore we obtain by \cref{prop homog spaces}	 for the above cases that
	\[
	DS_u\Bbbk[G/K]=DS_u\Bbbk[N_G(c)/(N_G(c)\cap K)].
	\]
	Further, in all of the above cases we have that our copy of $SL(1|1)$ lies in $K$ and is normalized by $N_G(c)$, so that $u$ will in fact act trivially on $N_G(c)/N_G(c)\cap K$; it follows that in the above cases we have:
	\[
	DS_u\Bbbk[G/K]=\Bbbk[N_G(c)/(N_G(c)\cap K)],
	\]
	and from this we obtain our explicit results given in the table below.
	
	We now consider three other cases:
	\[
	(\g\l(m|2n),\o\s\p(m|2n)), \ \ \ (\g\l(n|n),\p(n)), \ \ \ (\g\l(n|n),\q(n)).
	\]
	In the first two cases, we again choose a root subalgebra $\s\l(1|1)\sub\k$, and let $u\in\s\l(1|1)_{\ol{1}}$ with $c=u^2\neq0$	. Then we have that
	\[
	C_G(c)=GL(1|1)\times GL(1|1)\times G'.
	\]
	where $G'=GL(m-2|2n-2)$ for the first pair, and $G'=GL(n-2|n-2)$ for the second pair.   For $(\g\l(n|n),\q(n))$ we choose a copy of $\q(1)\sub\q(n)$ which is a factor subalgebra of a Cartan subalgebra, and let $u\in\q(1)_{\ol{1}}$ such that $c=u^2\neq0$.  Then we have that
	\[
	C_G(u)=GL(1|1)\times GL(n-1|n-1).
	\]
	For the above three cases, any element $c\in\t'$ with such a centralizer is unique up to action of the Weyl group of $K$ on $\t'$, scaling, and multiples of the center.  Using the same argument as above, we find that the vanishing set of $u$ is contained in $N_{G_0}(c)K$, and thus
	\[
	DS_u\Bbbk[G/K]=DS_u\Bbbk[N_G(c)/(N_G(c)\cap K)].
	\]
	Now in all the three above cases we in fact have $N_G(c)/(N_G(c)\cap K)=C_G(c)/(C_G(c)\cap K)$; we thus obtain the following:
	\begin{eqnarray*}
		DS_u\Bbbk[GL(m|2n)/OSp(m|2n)]& = &\Bbbk[GL(m-2|2n-2)/OSp(m-2|2n-2)]\\
		& \otimes & DS_u\Bbbk[GL(1|1)^2/GL(1|1)],
	\end{eqnarray*}
	\begin{eqnarray*}
		DS_u\Bbbk[GL(n|n)/P(n)]& = &\Bbbk[GL(n-2|n-2)/P(n-2)]\\ 
		& \otimes & DS_u\Bbbk[GL(1|1)^2/GL(1|1)],
	\end{eqnarray*}
	and
	\[
	DS_u\Bbbk[GL(n|n)/Q(n)]=\Bbbk[GL(n-1|n-1)/Q(n-1)]\otimes DS_u\Bbbk[GL(1|1)/Q(1)].
	\]
	Now one may compute explicitly that $DS_u\Bbbk[GL(1|1)^2/GL(1|1)]\cong\mathbb{A}^{0|1}$, where $\mathbb{A}^{0|1}=\Spec \Bbbk[\xi]$ for some odd variable $\xi$, and $DS_u\Bbbk[GL(1|1)/Q(1)]\cong\Bbbk[x]/(x^2=1)=\Bbbk[\mathbb{S}^{0}]$.

\newpage

		\renewcommand{\arraystretch}{2.1}	

	\begin{thm}We have the following table of computations for $DS_u\Bbbk[G/K]$:
		\label{thm table}
		\[
		\begin{array}{|c|c|} 
			\hline
			G/K & \Spec DS_u\Bbbk[G/K] \\
			\hline
			\makecell{GL(m|n)/GL(m|n-r) \\ \hspace{5em} \times GL(r)} & \makecell{GL(m-1|n-1)/GL(m-1|n-r-1)\\ \hspace{5em}\times GL(r)} \\
			\hline
			\makecell{SOSp(m|2n)/SOSp(m-r|2n)\\ \hspace{5em} \times SO(r) \\ m-r>2} & \makecell{SOSp(m-2|2n-2)/SOSp(m-2-r|2n-2)\\ \hspace{6em}\times SO(r)} \\
			\hline
			\makecell{SOSp(m|2n)/SOSp(2|2n) \\ \hspace{5em} \times SO(m-2)} & \left(\makecell{SOSp(m-2|2n-2)/SO(m-2)\\ \hspace{5em} \times Sp(2n-2)}\right)\times\mathbb{S}^{0}\\
			\hline
			\makecell{SOSp(m|2n)/SOSp(m|2n-2r) \\ \times Sp(2r)} & \makecell{SOSp(m-2|2n-2)/SOSp(m-2|2n-2-2r)\\ \hspace{5em}\times Sp(2r)}\\
			\hline
			\makecell{SOSp(m|2n)/SOSp(m-1|2n-2r) \\ \times SOSp(1|2r) \\ m\geq 4, r\leq n-1} & \makecell{SOSp(m-2|2n-2)/SOSp(m-3|2n-2-2r)\\ \hspace{5em}\times SOSp(1|2r)}\\
			\hline
				\makecell{SOSp(3|2n)/SOSp(2|2n-2r) \\ \times SOSp(1|2r) \\ r\leq n-1} & \makecell{SOSp(1|2n-2)/Sp(2n-2-2r)\\ \hspace{5em}\times SOSp(1|2r)}\times\mathbb{S}^{0}\\
			\hline
			GL(m|2n)/OSp(m|2n) & GL(m-2|2n-2)/OSp(m-2|2n-2)\times \mathbb{A}^{0|1}\\
			\hline
			SOSp(2m|2n)/GL(m|n) & \left(SOSp(2m-2|2n-2)/GL(m-1|n-1)\right)\times\mathbb{S}^{0} \\
			\hline
			GL(n|n)/Q(n) &  \large \left(GL(n-1|n-1)/Q(n-1)\right)\times\mathbb{S}^{0}\\
			\hline 
			GL(n|n)/P(n) &  \left(GL(n-2|n-2)/P(n-2)\right)\times\mathbb{A}^{0|1}\\
			\hline 
			P(n)/GL(r|n-r), 0<r<n & P(n-2)/GL(r-1|n-r-1)\\
			\hline
			D(2,1;\alpha)/SOSp(2|2)\times SO(2) & \mathbb{S}^0 \\
			\hline
			G(1|2)/[D(2,1;3)/z] & SL(2)/\mathbb{G}_m\\
			\hline
			G(1|2)/\left[(SpinSp(3|2)\times SL(2))/\langle(-1,-1)\rangle\right]  & \Spec\Bbbk \\
			\hline
			AB(3,1)/\left[\G_m\times (SpinSp(2|4)/(\pm1))\right] & \left(SL(3)/GL(2)\right)\times\mathbb{S}^0\\
			\hline
			AB(3,1)/SL(1|4) & \mathbb{S}^{0} \\
			\hline
			 AB(3,1)/\left[(D(2,1;2)\times SL(2))/\langle z'\rangle\right] & SL(3)/GL(2) \\
			\hline
		\end{array}
		\]	
	\end{thm}
	
	\renewcommand{\arraystretch}{1}
	
	As a matter of explanation for the elements $z,z'$, we note that the center of $D(2,1;\alpha)$ is a Klein 4 group, where the three nontrivial elements are naturally indexed by $1,\alpha$, and $1-\alpha$. 	In the row where we have $G(1|2)/[D(2,1;3)/z]$, the $z$ is the central element of $D(2,1;3)$ corresponding to $-4$.  In the last row of the above table, $z'$ is the element of $D(2,1;2)\times SL(2)$ given by $(x,-1)$, where $-1\in SL_2$, and $x\in D(2,1;2)(\C)$ is the central element corresponding to $-3$.  
	
	\begin{remark}
		For a smooth affine $G$-supervariety $X$, write $D^{G}(X)$ for the algebra of $G$-equivariant differential operators on $X$.  Then if $u\in\g_{\ol{1}}^{hom}$ satisfies the hypothesis of Corollary \ref{cor_z=y}, $Y$ will have a natural action by $\widetilde{G_u}$, and we will obtain a natural algebra morphism:
		\[
		D^G(X)\to D^{\widetilde{G_u}}(Y).
		\] 
		This map would be especially interesting to study in one of the cases above, when $G/K$ is a supersymmetric space.
	\end{remark}

	\begin{remark}
		To see why the condition (*) is important, consider the supersymmetric space $G/K=GL(3|3)/(GL(1|1)\times GL(2|2))$, and let $u\in\g\l(1|1)\sub\k$ be a generic rank one element lying in the $\g\l(1|1)_{\ol{1}}$ factor of $\k$.  Then there is an element $g$ of $G(\Bbbk)$, coming from the Weyl group, which has that $\Ad(g^{-1})(u)$ lies in the $\g\l(2|2)$ factor of $\k$.  Thus $gK\in Z(u)$; however one can check that $u$ does not satisfy condition (3) of Theorem \ref{main_thm} at $g$.  
	\end{remark}

	\subsection{Quotients by Levi subgroups of \texorpdfstring{$GL(n|n)$}{GL(n|n)}, \texorpdfstring{$P(n)$}{P(n)}, and \texorpdfstring{$Q(n)$}{Q(n)}}
	
	\begin{thm}\label{theorem gl homog space}
		Let $u=\begin{bmatrix} 	0 & I_n\\ \lambda I_n & 0	\end{bmatrix}\in\g\l(n|n)$ where $\lambda\in\Bbbk$, and consider the subgroup
		\[
		K=GL(r_1|r_1)\times\cdots\times GL(r_k|r_k),
		\]
		where $\sum\limits_i r_i=n$.  Then we have an isomorphism
		\begin{eqnarray*}
			DS_u[GL(n|n)/GL(r_1|r_1)\times\cdots\times GL(r_k|r_k)]& \cong & H_{dR}^\bullet(GL(n)/GL(r_1)\times\cdots\times GL(r_k)),
		\end{eqnarray*}
	where $H_{dR}^\bullet(-)$ denotes the de Rham cohomology.
	\end{thm}
	
	\begin{proof}
		Let $H\sub GL(n|n)$ be the subgroup with $\h=\g_{-1}\oplus\g_{0}$ and $H_0\cong GL(n)$ embedded in $GL(n|n)$ as the matrices of the form
		\[
		\begin{bmatrix}
			A & 0\\0 & A
		\end{bmatrix}
		\]
		for $A\in GL(n)$.  Let $Y=H/(K\cap H)$; then we claim that the vanishing set of $u$ is exactly $Y_0$.  Indeed, suppose that $u$ vanishes at $g^{-1}K$, i.e. $\Ad(g)(u)\in\k$, and write $g=(g_1,g_2)\in GL(n)\times GL(n)$.  Then we can write $g=(k,I_n)(g_2,g_2)$ where $k=g_1g_2^{-1}$, and now $(g_2,g_2)\in H_0$.    Observe that
		\[
		\Ad(k,I_n)(u)=\begin{bmatrix}
			0 & k\\ \lambda k^{-1} & 0
		\end{bmatrix}.
		\]
		Thus we must have $k\in GL(r_1)\times\cdots\times GL(r_k)$ so that $(k,I_n)\in K$.  Therefore $g^{-1}K=h^{-1}K$, where $h=(g_2,g_2)$.  Now from the \cref{main_thm}, we obtain that
		\[
		DS_u\Bbbk[GL(n|n)/K]=DS_u\Bbbk[Y].
		\]
		We finish by observing that $u$ defines an everywhere vanishing vector field on $Y$, and satisfies the hypotheses of \ref{lemma odd cotang}.  Since $Y_0=GL(n)/GL(r_1)\times\cdots\times GL(r_k)$, this completes the proof.
	\end{proof}
	
	\begin{thm}
		Let $G=P(n)$, $K=P(r_1)\times\cdots\times P(r_k)$, where $\sum\limits_ir_i=n$.  Let 
		\[
		u=\begin{bmatrix}
			0 & I_n\\0 & 0
		\end{bmatrix}\in\p(n)_{\ol{1}}; 
		\]
		if $n$ is even, then further set
		\[
		v=\begin{bmatrix}0 & 0\\ A & 0	\end{bmatrix}\in\p(n)_{\ol{1}},
		\]
		where $A$ is an invertible skew-symmetric matrix.  Then
		\[
		DS_u\Bbbk[G/K]\cong H^\bullet_{dR}(O(n)/O(r_1)\times\cdots\times O(r_k)),
		\]
		and when $n,r_1,\dots,r_k$ are all even,
		\[
		DS_v\Bbbk[G/K]\cong H^\bullet_{dR}(Sp(n)/Sp(r_1)\times\cdots\times Sp(r_k)).
		\]
	\end{thm}
	
	\begin{proof}
		In the first case, consider the subgroup $H\sub G$ with $H_0=O(n)\sub GL(n)$, and $\h_{\ol{1}}=\g_{-1}$.  Set $Y=H/H\cap K\sub G/K$.  In the second case, consider the subgroup $H'\sub G$ with $H'_0=Sp(\Bbbk^n,A)\sub GL(n)$, the matrices preserving $A$, and $\h_{\ol{1}}=\g_{1}$.  Then set $Y=H/H\cap K\sub G/K$. Now using the equivalence of all symmetric, resp. alternating nondegenerate bilinear forms on $\Bbbk^{n}$, we can prove that the vanishing set of $u$, resp. $v$ in $G/K$ is given by $Y_0$.  From here, we use that $[u,-]:\h_{\ol{1}}\to\h_{\ol{0}}$, resp. $[v,-]:\h'_{\ol{1}}\to\h'_{\ol{0}}$ are isomorphisms, and \cref{lemma odd cotang}.
	\end{proof}

	\begin{thm}
		Consider the element in $\q(n)$ given by:
		\[
		u=\begin{bmatrix}
			0 & I_n\\ I_n & 0
		\end{bmatrix}.
		\]
		Let $r_1,\dots,r_k$ be positive integers with $\sum\limits r_i=n$.  Then we have an isomorphism of algebras:
		\[
  		DS_u\Bbbk[Q(n)/Q(r_1)\times\dots\times Q(r_k)]\cong H_{dR}^\bullet(GL(n)/GL(r_1)\times\dots\times GL(r_k)).
		\]
	\end{thm}
	
	\begin{proof}
		We check the hypotheses of \cref{lemma odd cotang}; clearly $u^2$ acts trivially, and it commutes with $Q(n)_0$.  Since it normalizes $\k=\q(r_1)\times\dots\times\q(r_k)$, and defines an isomorphism $\q(n)_{\ol{1}}\to\q(n)_{\ol{0}}$, the map
		\[
		(\q(n)/\k)_{\ol{1}}\to(\q(n)/\k)_{\ol{0}}
		\]
		is an isomorphism. 
	\end{proof}

\section{Appendix: A root subgroup of \texorpdfstring{$D(2,1;\alpha)$}{D(2,1;a)}}

Consider the subgroup $H$ of $D(2,1;\alpha)$ with Lie superalgebra given by $\g_{\beta}\oplus\g_{-\beta}\oplus\mathfrak{t}$, where $\mathfrak{t}$ is a maximal torus and $\beta$ is an isotropic odd root.  Its Hopf superalgebra structure is given as follows: it has functions $\Bbbk[x_{1}^{\pm1},x_2^{\pm1},x_3^{\pm1},\xi,\eta]$, where the coproduct $\Delta$ is given by:
\[
\Delta(x_i)=x_i\otimes x_i+c_ix_i\eta\otimes x_i\xi, \ \ \ \Delta(\xi)=\xi\otimes 1+x_1x_2x_3\otimes \xi, \ \ \ \Delta(\eta)=\eta\otimes x_1x_2x_3+1\otimes \eta.
\]
where $c_1+c_2+c_3=0$, and explicitly we have $c_1=\alpha$, $c_2=1$, $c_3=-1-\alpha$.

Let $x\in\g_{\beta}$ correspond to the derivation $\d_{\xi}$ at the identity, and $y\in\g_{-\beta}$ correspond to the derivation $\d_{\eta}$ at the identity.  Then we obtain, by direct computation, that the adjoint vector field of $u=x+y$ is given by:
\[
(1-x_1x_2x_3)(\d_{\eta}-\d_{\xi})-(\xi+\eta)\sum\limits_ic_ix_i\d_{x_i}.
\]
We can compute its cohomology on $\Bbbk[H]$ as follows: set $\zeta=\xi+\eta$ and $\gamma=(\eta-\xi)/2$.  Then we may rewrite the above as
\[
D_1+D_2=(1-x_1x_2x_3)\d_{\gamma}-\zeta\sum\limits_ic_ix_i\d_{x_i}.
\]
where 
\[
D_1:=(1-x_1x_2x_3)\d_{\gamma}, \ \ \ \ D_2:=-\zeta\sum\limits_ic_ix_i\d_{x_i}.
\]
It is a straightforward computation that $[D_1,D_2]=0,$ so we can view this as computing the cohomology of the total complex of the following bicomplex:
\[
\xymatrix{
A_0 \ar[r]^{D_2}& A_0\langle\zeta\rangle\\
A_0\langle\gamma\rangle \ar[u]^{D_1} \ar[r]^{D_2}& A_0\langle\zeta\gamma\rangle\ar[u]_{D_1}
}
\] 
where $A_0=\Bbbk[x_1^{\pm1},x_2^{\pm1},x_3^{\pm1}]$.  Thus we may compute the cohomology via the spectral sequence by the filtration along the columns. Cohomology with respect to $D_1$ will give a Koszul complex, and taking its cohomology we obtain:
\[
\Bbbk[x_1^{\pm1},x_2^{\pm1},x_3^{\pm1},\zeta]/(1-x_1x_2x_3).
\]
The spectral sequence collapses on the second page, and since it is concentrated along a single row, it will in fact exactly compute the cohomology for us.  

Thus it remains to compute the cohomology of $D_2$; from the relation, we can forget about $x_3$, and so we obtain $\Bbbk[x_1^{\pm1},x_2^{\pm1},\zeta]$ with operator $D_2=\zeta(c_1x_1\d_{x_2}+c_2x_2\d_{x_2})$.  Recall that $c_1=\alpha$ and $c_2=1$.  Thus if $\alpha\notin\Q$, we obtain that the cohomology is $\Bbbk[\zeta]$, i.e. we have $\widetilde{H_u}\cong\mathbb{G}^{0|1}$.  

On the other hand, suppose that $\alpha=m/n\in\Q$ in reduced form.  Then the cohomology is given by $\Bbbk[z^{\pm1},\zeta]$, where $z=x_1^nx_2^{-m}$.  One can check that the Hopf algebra structure gives $\widetilde{H_u}\cong \mathbb{G}_m\times\mathbb{G}^{0|1}$.

\bibliographystyle{amsalpha}

\begin{thebibliography}{99999999}
	
	\bibitem[CCF]{carmeli2011mathematical}	C.~Carmeli, L.~Caston, and R.~Fioresi. \emph{Mathematical foundations of supersymmetry}, Vol.~15, European Mathematical Society (2011).
	
	\bibitem[CCMV]{CCMV} C.~Candu, T.~Creutzig, V.~Mitev, and V.~Schomerus.  {\em Cohomological reduction of sigma models}, Journal of High Energy Physics, Vol.~5 (2010): 1--39.
	
	\bibitem[CW]{CW} S.J.~Cheng and W.~Wang. \emph{Dualities and representations of Lie superalgebras}, American Mathematical Society (2012).
	
	\bibitem[Cos]{Cos} K.~Costello. {\em Notes on supersymmetric and holomorphic field theories in dimensions 2 and 4}, Pure and Applied Mathematics Quarterly, Vol.~9, no.~1 (2013): 73--165. 
	
	\bibitem[DS]{duflo2005associated} M.~Duflo and V.~Serganova. \emph{On associated variety for {Lie} superalgebras}, arXiv preprint:0507198.
	
	\bibitem[EAHS]{entova2018deligne}  I.~Entova-Aizenbud, V.~Hinich, and V.~Serganova. \emph{{Deligne}	categories and the limit of categories Rep($\mathfrak{gl}(m|n)$)}, International Mathematics Research Notices, Vol.~15 (2018): 4602--4666.
	
	\bibitem[EAS]{entovaaizenbud2019dufloserganova}	I.~Entova-Aizenbud and V.~Serganova. \emph{Duflo-Serganova functor and		superdimension formula for the periplectic {Lie} superalgebra}, Algebra \& Number Theory, Vol.~16, no.~3, (2022): 697--729.
	
	\bibitem[GH]{gorelik2020semisimplicity}	M.~Gorelik and T.~Heidersdorf. \emph{Semisimplicity of the $DS$	functor for the orthosymplectic {Lie} superalgebra}, Advances in Mathematics, Vol.~394 (2022): 108012.
	
	\bibitem[GHSS]{GHSS}M.~Gorelik, C.~Hoyt, V.~Serganova, and A.~Sherman. \emph{The {Duflo-Serganova} functor, vingt ans apres}, Journal of the Indian Institute of Science (2022): 1---40.
	
	\bibitem[GK]{GK} M.~Grigoriev and A.~Kotov. \emph{Gauge PDE and AKSZ--type Sigma Models: LMS/EPSRC Durham Symposium on Higher Structures in M--Theory.} Fortschritte der Physik, Vol.~67, no.~8-9 (2019): 1910007.
	
	
	\bibitem[Har]{hartshorne2013algebraic}  R.~Hartshorne. \emph{Algebraic geometry}, Vol.~52, Springer Science \&	Business Media (2013).
	
	\bibitem[HW]{heidersdorf2014cohomological}  T.~Heidersdorf and R.~Weissauer. \emph{Cohomological tensor functors on representations of the general linear supergroup},  American Mathematical Society, Vol.~270, no.~1320 (2021).
	
	\bibitem[Hop]{hopf1964topologie}  H.~Hopf. \emph{{\"U}ber die Topologie der Gruppen-Mannigfaltigkeiten und ihrer Verallgemeinerungen}, Selecta Heinz Hopf, Springer, (1964): 119--151.
	
	\bibitem[I]{I72}  B.~Iversen. \emph{A fixed point formula for action of tori on algebraic varieties}, Inventiones mathematicae, Vol.~16 no.~3 (1972): 229--236.
	
	\bibitem[M]{M} Y.I.~Manin. \emph{Gauge field theory and complex geometry},  Springer Science \& Business Media, Vol.~28 (2013).
	
	\bibitem[MaT]{MaT} A. Masuoka and Y. Takahashi. {\it Geometric construction of quotients G/H in supersymmetry}, Transformation Groups Vol.~26, no.~1 (2021): 347--375.
	
	\bibitem[Mu]{Mu} I.M.~Musson. \emph{Lie superalgebras and enveloping algebras}, American Mathematical Society, Vol.~131 (2012).
	
	\bibitem[Sch]{Sch89} T.~Schmitt. \emph{Regular sequences in $\Z_2$-graded commutative algebra}, Journal of Algebra, Vol.~124, no.~1 (1989): 60--118.
	
	\bibitem[SZ97]{schwarzzaborsupersymlocal}  A.~Schwarz and O.~Zaboronsky. \emph{Supersymmetry and localization}, Communications in Mathematical Physics, Vol.~183 (1997): 463--476.
	
	\bibitem[Ser1]{serganova2011superdimension} 	V.~Serganova. \emph{On the superdimension of an irreducible representation of
		a basic classical {Lie} superalgebra}, Supersymmetry in Mathematics and	Physics, Springer (2011): 253--273.
	
	\bibitem[Ser2]{Ser2} V~ Serganova. \emph{Quasireductive supergroups}, New developments in Lie theory and its applications, Vol.~544 (2011): 141--159.
	
	\bibitem[Ser3]{Ser3} V.~Serganova. \emph{Classification of real simple {Lie} superalgebras and symmetric superspaces}, Functional Analysis and its Applications, Vol.~17, no.~3 (1983): 200--207.
	
	\bibitem[SS]{SS22} V.~Serganova and A.~Sherman. \emph{Splitting quasireductive supergroups and volumes of supergrassmannians}, Algebraic Geometry and Physics (to appear). 
	
	\bibitem[SV]{SV} V.~Serganova and D.~Vaintrob.  \emph{Localization for CS supermanifolds and volume of homogeneous superspaces}, arXiv preprint: 2212.07503.
	
	\bibitem[She1]{sherman2021spherical}  A.~Sherman. \emph{Spherical supervarieties}, Annales de l'Institut Fourier, Vol.~71 (2021): 1449--1492.
	
	\bibitem[She2]{sherman2022symmetriesDS}	A.~Sherman.  \emph{On symmetries of the Duflo-Serganova functor}, Israel Journal of Mathematics (to appear).
	
	\bibitem[Vai]{vaintrob1996normal} A~Vaintrob. \emph{Normal forms of homological vector fields}, Journal of Mathematical Sciences, Vol.~82, no.~6 (1996): 3865--3868.
	
	\bibitem[VMP]{voronov1990elements}	A.A.~Voronov, Y.I.~Manin, and I.B.~Penkov. \emph{Elements of supergeometry}, Journal of Soviet Mathematics Vol.~51, no.~1 (1990): 2069--2083.
	
	\bibitem[W]{W} E.~Witten. {\em Topological quantum field theory}, Communications in Mathematical Physics, Vol.~117, no.~3 (1988): 353--386.
	
	
\end{thebibliography}

	\textsc{\footnotesize Vera Serganova, Dept. of Mathematics, University of California at Berkeley, Berkeley, CA 94720} 
	
	\textit{\footnotesize Email address:} \texttt{\footnotesize serganov@math.berkeley.edu}
	
	\textsc{\footnotesize Alexander Sherman, Dept. of Mathematics, University of Sydney, Camperdown, Australia} 
	
	\textsc{\footnotesize and Dept. of Mathematics, Ben Gurion University, Beer-Sheva,	Israel} 
	
	\textit{\footnotesize Email address:} \texttt{\footnotesize xandersherm@gmail.com}

\end{document}